\documentclass[]
{amsart}
\usepackage{amsmath, amssymb, amscd, amsthm, amsfonts, amstext,
  amsbsy,  mathrsfs, mathtools, stmaryrd, enumitem, comment}

\DeclareRobustCommand{\SkipTocEntry}[5]{}

\makeatletter
\@namedef{subjclassname@2020}{\textup{2020} Mathematics Subject Classification}
\makeatother

\usepackage[textsize=footnotesize,color=yellow!20, bordercolor=white]{todonotes}

\usepackage{manfnt} 

\usepackage[scr=esstix]{mathalfa} 

\usepackage
[a4paper, margin=1.3in]{geometry}

\usepackage{url} 
\usepackage{hyperref}
\definecolor{dblue}{rgb}{0,0,0.70}
\hypersetup{
	unicode=true,
	colorlinks=true,
	citecolor=dblue,
	linkcolor=dblue,
	anchorcolor=dblue, 
	urlcolor = dblue
}

\makeatletter
\renewcommand{\tocsection}[3]{%
  \indentlabel{\@ifnotempty{#2}{\bfseries\ignorespaces#1 #2\quad}}\bfseries#3}
\renewcommand{\tocsubsection}[3]{%
  \indentlabel{\@ifnotempty{#2}{\ignorespaces#1 #2\quad}}#3}
\renewcommand{\tocsubsubsection}[3]{%
  \indentlabel{\hspace{30pt}\@ifnotempty{#2}{\ignorespaces#1 #2\quad}}#3}

\def\@tocline#1#2#3#4#5#6#7{\relax
  \ifnum #1>\c@tocdepth 
  \else
    \par \addpenalty\@secpenalty\addvspace{#2}%
    \begingroup \hyphenpenalty\@M
    \@ifempty{#4}{%
      \@tempdima\csname r@tocindent\number#1\endcsname\relax
    }{%
      \@tempdima#4\relax
    }%
    \parindent\z@ \leftskip#3\relax \advance\leftskip\@tempdima\relax
    \rightskip\@pnumwidth plus1em \parfillskip-\@pnumwidth
    #5\leavevmode\hskip-\@tempdima{#6}\nobreak
    \leaders\hbox{$\m@th\mkern \@dotsep mu\hbox{.}\mkern \@dotsep mu$}\hfill
    \nobreak
    \hbox to\@pnumwidth{\@tocpagenum{\ifnum#1=1\bfseries\fi#7}}\par
    \nobreak
    \endgroup
  \fi}
\AtBeginDocument{%
\expandafter\renewcommand\csname r@tocindent0\endcsname{0pt}
}
\def\l@subsection{\@tocline{2}{0pt}{2.5pc}{5pc}{}}
\makeatother

\newcommand{\CC}{\mathbb{C}}
\newcommand{\RR}{\mathbb{R}}

\newcommand{\PP}{\mathbb{P}}
\newcommand{\QQ}{\mathbb{Q}}
\newcommand{\SSS}{\mathbb{S}}
\newcommand{\BB}{\mathbb{B}}

\newcommand{\HH}{\mathbb{H}}

\newcommand{\id}{\operatorname{id}}
\newcommand{\cof}{\operatorname{cof}}

\newcommand{\pow}{\mathcal{P}}

\newcommand{\p}{\operatorname{p}}

\newcommand{\Card}{{\sf Card}}

\newcommand{\AD}{{\sf AD}}
\newcommand{\ZFC}{{\sf ZFC}}
\newcommand{\ZF}{{\sf ZF}}
\newcommand{\AC}{{\sf AC}}
\newcommand{\DC}{{\sf DC}}
\newcommand{\HOD}{{\sf HOD}}
\newcommand{\Ord}{\mathsf{Ord}}
\newcommand{\FA}{{\sf FA}}

\newcommand{\A}{{\sf{A}}}
\newcommand{\OD}{{\mathsf{OD}}}

\newcommand{\Fun}{\mathrm{Fun}_{<\omega}}
\newcommand{\Func}{\mathrm{Fun}}
\newcommand{\Reg}{\mathsf{Reg}}
\newcommand{\SucCard}{\mathsf{SucCard}}

\newcommand{\ran}{\mathrm{ran}}
\newcommand{\dom}{\mathrm{dom}}

\newcommand{\Add}{{\rm Add}}

\newcommand{\supp}{\mathrm{supp}} 
\newcommand{\Col}{\mathrm{Col}} 

\newcommand{\Lev}{\mathrm{Lev}}

\newcommand{\cC}{\mathcal{C}}

\newcommand{\li}{\leq_i}
\newcommand{\one}{\mathbf{1}}
\newcommand{\rank}{\mathrm{rank}}
\newcommand{\tr}{\mathrm{tr}}
\newcommand{\GA}{\mathsf{A}}

\newcommand{\cc}{\mathbf{c}} 
\newcommand{\jj}{\mathbf{m}} 
\newcommand{\bb}{\mathbf{b}} 
\newcommand{\dd}{\mathbf{d}} 
\newcommand{\pair}{\mathrm{pair}} 
\newcommand{\red}{\mathrm{red}}
\newcommand{\foot}{\mathrm{f}}
\newcommand{\Bore}{$\mathrm{Borel^*}$} 
\newcommand{\WO}{\mathrm{WO}}

\newcommand{\cf}{\mathrm{cof}}
\newcommand{\cfdash}{\mathrm{cof'}}
\newcommand{\otp}{\text{otp}}

\newtheorem{theorem}{Theorem}[section]
\newtheorem{lemma}[theorem]{Lemma}
\newtheorem{corollary}[theorem]{Corollary}

\newtheorem{problem}[theorem]{Problem}

\theoremstyle{definition}

\newtheorem*{claim*}{Claim}

\newtheorem*{subclaim*}{Subclaim}
\newtheorem{definition}[theorem]{Definition}

\theoremstyle{remark}
\newtheorem{remark}[theorem]{Remark}

\newtheorem{notation}[theorem]{Notation}
\newtheorem{assumption}[theorem]{Assumption}

\newtheorem*{case*}{Case}

\newcommand\nnfootnote[1]{%
  \begin{NoHyper}
  \renewcommand\thefootnote{}\footnote{#1}%
  \addtocounter{footnote}{-1}%
  \end{NoHyper}
}

\newenvironment{enumerate-(a)}{\begin{enumerate}[label={\upshape (\alph*)}, leftmargin=2pc]}{\end{enumerate}}
\newenvironment{enumerate-(a)-r}{\begin{enumerate}[label={\upshape (\alph*)}, leftmargin=2pc,resume]}{\end{enumerate}}
\newenvironment{enumerate-(A)}{\begin{enumerate}[label={\upshape (\Alph*)}, leftmargin=2pc]}{\end{enumerate}}
\newenvironment{enumerate-(A)-r}{\begin{enumerate}[label={\upshape (\Alph*)}, leftmargin=2pc,resume]}{\end{enumerate}}
\newenvironment{enumerate-(i)}{\begin{enumerate}[label={\upshape (\roman*)}, leftmargin=2pc]}{\end{enumerate}}
\newenvironment{enumerate-(i)-r}{\begin{enumerate}[label={\upshape (\roman*)}, leftmargin=2pc,resume]}{\end{enumerate}}
\newenvironment{enumerate-(I)}{\begin{enumerate}[label={\upshape (\Roman*)}, leftmargin=2pc]}{\end{enumerate}}
\newenvironment{enumerate-(I)-r}{\begin{enumerate}[label={\upshape (\Roman*)}, leftmargin=2pc,resume]}{\end{enumerate}}
\newenvironment{enumerate-(1)}{\begin{enumerate}[label={\upshape (\arabic*)}, leftmargin=2pc]}{\end{enumerate}}
\newenvironment{enumerate-(1)-r}{\begin{enumerate}[label={\upshape (\arabic*)}, leftmargin=2pc,resume]}{\end{enumerate}}
\newenvironment{itemizenew}{\begin{itemize}[leftmargin=2pc]}{\end{itemize}}

\AtBeginDocument{%
   \def\MR#1{}
}

\begin{document}

\author{Daisuke Ikegami}
\address[Daisuke Ikegami]{Institute of logic and cognition, department of philosophy, Sun Yat-sen University, Xichang hall 602, 135 Xingang west street, Guangzhou, 510275 China}
\email{ikegami@mail.sysu.edu.cn}


\author{Philipp Schlicht} 
\address[Philipp Schlicht]{Dipartimento di ingegneria dell'informazione e scienze matematiche, Università di Siena, via Roma 56, 53100 Siena, Italy} 
\email{philipp.schlicht@unisi.it}


\title[Forcing over choiceless models and generic absoluteness]{Forcing over choiceless models \\ and generic absoluteness} 


\keywords{Choiceless set theory, forcing, generic absoluteness} 
\subjclass[2020]{(Primary) 03E25; (Secondary) 03E57, 03E35, 03E17}

\makeatletter
\def\blfootnote{\gdef\@thefnmark{}\@footnotetext}
\makeatother

\nnfootnote{The authors are grateful to W. Hugh Woodin for discussions and for permission to include his Theorem \ref{separate random Cohen} and the joint Theorem \ref{abs implies countable}. 
They further thank Arthur Apter 
and Peter Koepke for discussions on set theory without the axiom of choice during initial stages of this project, and Elliot Glazer, Asaf Karagila and John Truss for comments.} 
\nnfootnote{The first-listed author thanks the Japan Society for the Promotion of Science (JSPS) for its generous support through the grants with JSPS KAKENHI Grant Number 15K17586 and 19K03604. He is also grateful to the Sumitomo Foundation for its generous support through a Grant for Basic Science Research. 
This research was funded in whole or in part by EPSRC grant number EP/V009001/1 for the second-listed author and has received funding from the European Union's Horizon 2020 research and innovation programme under his 
Marie Sk\l odowska-Curie grant agreement No 794020. 
He gratefully acknowledges the support of INdAM-GNSAGA. 
For the purpose of open access, the authors have applied a \lq Creative Commons Attribution' (CC BY) public copyright licence to any Author Accepted Manuscript (AAM) version arising from this submission.}

\begin{abstract} 
We develop a toolbox for forcing over arbitrary models of set theory without the axiom of choice. 
In particular, we introduce a variant of the countable chain condition and prove an iteration theorem 
that applies to many classical forcings such as Cohen forcing and random algebras. 
Our approach sidesteps the problem that forcing with the countable chain condition can collapse $\omega_1$ by a recent result of 
Karagila and Schweber. 





Using this, we show that adding many Cohen reals and random reals leads to different theories. 
This result is due to Woodin. 
Thus one can always change the theory of the universe by forcing, just like the continuum hypothesis and its negation can be obtained by forcing over arbitrary models with choice.

We further study principles stipulating that the first-order theory of the universe remains the same in all generic extension by a fixed class of forcings. 
Extending a result of Woodin, we show that even for quite restricted classes such as the class of all finite support products of Cohen forcing or the class of all random algebras, this principle implies that all infinite cardinals have countable cofinality. 
\end{abstract}

\maketitle

\setcounter{tocdepth}{2}
\tableofcontents


\section{Introduction} 


Forcing over choiceless models of set theory appears in the context of forcing over $L(\RR)$ 
in work of Steel, Van Wesep \cite{steel1982two}, Woodin \cite{woodin2013axiom}, Laflamme \cite{laflamme1989forcing}, Di Prisco and Todor\v{c}evi\'c \cite{di1998perfect}. 
The idea of Woodin's $\PP_{\mathrm{max}}$-forcing \cite{woodin2013axiom} and its variants is to obtain consistency results not accessible by forcing over models of choice. 
Blue, Larson and Sargsyan \cite{Nairian-models,failure-of-squares} force over models of determinacy beyond $L(\RR)$ to obtain the failure of square principles. 
Larson and Zapletal \cite{larson2020geometric} build a theory for separating consequences of choice. 
General results about forcing over arbitrary choiceless models appeared in studies of Monro \cite{monro1983generic} and recent work of Karagila, Schlicht \cite{karagila2020have} and Usuba \cite{usuba2018choiceless}. 
Most basic results in forcing such as the forcing theorem go through in $\ZF$ without use of the axiom of choice.\footnote{An exception is maximality or fullness, i.e., the statement: if an existential formula $\exists x\ \varphi(x,\tau)$ is forced, then there exists a name $\sigma$ such that $\varphi(\sigma,\tau)$ is forced.} 
Regarding the preservation of cardinals, there is a serious problem because over choiceless models, 
even simple forcings can collapse cardinals. 
For instance, both $\sigma$-closed forcings and forcings with the countable chain condition can collapse $\omega_1$, the latter by a recent result by Karagila and Schweber \cite{karagila2022choiceless}.\footnote{They find a model and a forcing such that each antichain is countable, but $\omega_1$ is collapsed. 
In the present paper, c.c.c. means the weaker condition that there exist no antichains of size $\omega_1$.} 
While even the the most basic $\sigma$-closed forcings collapse $\omega_1$ if $\omega_1$ is singular, various well-known c.c.c. forcings do preserve cardinals. 
For instance, this is the case for finite support products of Cohen forcing, finite support iterations of random forcing and Hechler forcing and random algebras on all cardinals. 
We show this in this paper. 
We choose a general approach by introducing a variant of the countable chain condition that can be iterated. 
In models of the axiom of choice, this notion is equivalent to the c.c.c. 
The more general versions are called $\theta$-narrow and uniformly $\theta$-narrow, respectively, for any infinite ordinal $\theta$. 
We show in Lemma \ref{narrow pres card} and Theorem \ref{iteration theorem}: 

\begin{theorem}\ 
\begin{enumerate-(1)} 
\item 
Every $\theta$-narrow forcing preserves all cardinals and cofinalities ${>}\theta$. 

\item 
Any uniform iteration of $\theta$-narrow forcings with finite support is again uniformly $\theta$-narrow. 

\end{enumerate-(1)} 
\end{theorem} 

The notion of a uniform iteration of $\theta$-narrow forcings is explained below. 
As an application, we will see that a mix of any of the above forcings can be iterated with finite support while preserving all cardinals and cofinalities. 
In order to apply the previous result to random algebras, we show in Theorem \ref{random complete}: 

\begin{theorem} 
The random algebra on any number of generators is complete. 
\end{theorem} 

A special case of the above iteration theorem is that any uniform iteration of $\sigma$-linked forcings preserves all cardinals and cofinalities in Corollary \ref{iteration linked}. 
Here, uniform means that the iteration comes with a sequence of names for linking functions. 

We use these results to study the effect of products and iterations of the above forcings on choiceless models. 
By a Cohen model, we mean an extension by a finite support product of Cohen forcing of length at least $\omega_2$. 
A random model is an extension by a random algebra with at least $\omega_2$ generators. 
We separate Cohen from random models in Theorem \ref{paper failure absoluteness}. 
This result was proved by Woodin. 

\begin{theorem}[Woodin] 
\label{separate random Cohen} 
Cohen and random models have different theories. 
\end{theorem} 

In particular, there is a first-order sentence that can be forced true or false over any choiceless model, as you like. 
Using this, we study the principle $\A_\cC$ 
which states $V$ is elementarily equivalent to all its generic extensions by forcings in a class $\cC$. 
The previous result shows that $\A_{\cC}$ fails for the class $\cC$ of all random algebras and all finite support products of Cohen forcings. 
In the language of the modal logic of forcing \cite{hamkins2008modal}, we find a \emph{switch} with respect to this class. 
Recall that a switch is a sentence that can be forced both true or false over any generic extension of $V$. 
For example, the continuum hypothesis and the existence of Suslin trees are switches for models of $\ZFC$. 


We then study more restrictive absoluteness principles for each of 
the classes $\CC^*$ of all finite support products of Cohen forcings, $\RR_*$ of all random algebras and $\HH^{(*)}$ of finite support iterations of Hechler forcing of arbitrary length. 
If there exists an uncountable regular cardinal, then within each of these classes the generic extension can detect differences in the the length of the product, number of generators or length of the iteration in the ground model. 
Using this, we will show in Theorem \ref{A implies singular} and Corollary \ref{abs H singular} that each of these principles implies that all infinite cardinals have countable cofinality. 
The next result was proved by Woodin for the class $\CC^{*}$. 

\begin{theorem}[joint with Woodin] 
\label{abs implies countable} 
Let $\A$ denote any of the principles $\A_{\CC^*}$, $\A_{\RR_*}$ or $\A_{\HH^{(*)}}$. 
If $\A$ holds, then all infinite cardinals have countable cofinality. 
\end{theorem} 

For random algebras and products of Cohen forcing, this is proved in Section \ref{generic absoluteness} using a new cardinal characteristic $\jj$. 
Regarding Hechler forcing, let $\HH^{(\kappa)}$ denote the finite support iteration of Hechler forcing of length $\kappa$. 
By the above iteration theorem, we know that $\HH^{(\kappa)}$ preserves all cardinals and cofinalities. 
The case for Hechler forcing of the previous theorem uses the next result that is proved in Theorems \ref{bounding regular} and \ref{bounding}: 


\begin{theorem} 
For any uncountable cardinal $\kappa$, $\HH^{(\kappa)}$ forces that the bounding number equals $\max(\cof(\kappa),\omega_1)$. 
\end{theorem} 

In a well known model constructed by Gitik, all infinite cardinals have countable cofinality. 
However, Theorem \ref{A fails Gitik} shows that the above principles do not hold there: 

\begin{theorem}
Let $\A$ denote any of the principles $\A_{\CC^*}$, $\A_{\RR_*}$ or $\A_{\HH^{(*)}}$. 
$\A$ fails in Gitik's model from \cite[Theorem I]{MR576462}.  
\end{theorem} 




It is open whether $\A$ is consistent for any of the above choices. 
Our results in Section \ref{cons of abs for ckappa} provide properties that a model of $\A$ must have. 
Regarding its consistency strength, note that Gitik's construction starts from a proper class of strongly compact cardinals. 
Moreover, it follows from the previous theorem and a result of Busche and Schindler \cite[Theorem 1.5]{busche2009strength} that $\A$ has at least the consistency strength of $\AD^{L(\RR)}$. 
It therefore has at least the consistency strength of infinitely many Woodin cardinals.

\section{Preliminaries} 

\subsection{Notation}

Write $A\li B$ if there exists an injective function $f\colon A\rightarrow B$ and $A\leq_s B$ if there exists a surjective function $g\colon B\rightarrow A$ or $A$ is empty.\footnote{The standard notation for $\leq_i$ and $\leq_s$ is $\leq$ and $\leq^*$, respectively. However, this conflicts with other notation in this paper.} 
A \emph{partial function} from $A$ to $B$ is denoted $f\colon A \rightharpoonup B$. 
Basic open sets in Cantor space $2^\omega=\{f \mid f\colon \omega\rightarrow 2\}$ are denoted $N_t=\{x\in 2^{\omega} \mid t\subseteq x\}$, where $t\in 2^{<\omega}$. 

Let $\p\colon \Ord\times \Ord\rightarrow \Ord$ denote the standard pairing function.\footnote{An ordinal $\alpha>0$ is closed under $\p$ if and only if $\alpha$ is multiplicatively closed, i.e., $\beta\cdot \gamma<\alpha$ for all $\beta,\gamma<\alpha$.} 
The \emph{rank} $\rank(x)$ of a set $x$ is the least $\alpha\in\Ord$ with $x\in V_{\alpha+1}$. 
For ordinals $\alpha$ and $\beta$, $\pow_\alpha(\beta)$ denotes the set of subsets of $\beta$ of order type ${<}\alpha$. 
A \emph{cardinal} is an ordinal $\kappa$ that is not the surjective image of any $\alpha<\kappa$. 
$\Card$, $\SucCard$ and $\Reg$ denote the classes of infinite cardinals, infinite successor cardinals and infinite regular cardinals, respectively. 
$\kappa$ and $\lambda$  
denote infinite cardinals, unless stated otherwise. 
$\kappa$ is called a \emph{$\lambda$-strong limit} if for all $\nu<\kappa$, there does not exist a surjection from $\nu^\lambda$ onto $\kappa$. 
$\kappa$ is called \emph{$\lambda$-inaccessible} if it is a $\lambda$-strong limit and $\cof(\kappa)>\lambda$.

A \emph{forcing} is a set $\PP$ with a \emph{quasiorder}\footnote{A quasiorder is a transitive reflexive relation.} $\leq$ on $\PP$ and a largest element $\one_\PP$. 
We often identify $\PP$ and $(\PP,\leq,\one_\PP)$. 
The \emph{discrete partial order} on a set is the one with no relation between distinct elements. 
Conditions $p,q \in \PP$ are \emph{compatible}, denoted $p\parallel q$, if there exists some $r\leq p,q$. 
Otherwise $p$ and $q$ are \emph{incompatible}, denoted $p \perp q$. 
If $(\PP,\leq)$ and $(\QQ,\leq)$ are forcings, a \emph{$\parallel$-homomorphism} from $\PP$ to $\QQ$ is a function $f\colon \PP\rightarrow \QQ$ such that for all $p,q\in \PP$, $p\parallel q$ implies $f(p) \parallel f(q)$.  
The notions of \emph{$\leq$-homomorphism} and \emph{$\perp$-homomorphism} are defined similarly. 

The \emph{Boolean completion} $\BB(\PP)$ of a forcing $\PP$ is the set of all regular open subsets of $\PP$, ordered by inclusion.\footnote{It is easy check that $\BB(\PP)$ is isomorphic to $\BB(\PP_{\mathrm{sep}})$, where $\PP_{\mathrm{sep}}$ denotes the separative quotient of $\PP$.} 
It comes with a canonical $\leq$- and $\perp$-homomorphism $\iota:=\iota_\PP \colon \PP\rightarrow \BB(\PP)$ with dense range in $\BB(\PP)$, where $\iota(p)=\{q\in \PP \mid  \forall r\leq q\ r\leq p\}$. 
We will use the notation $p_\iota:=\iota(p)$. 
For any subset $A$ of $\PP$, we call $\sup(A):=\sup(\iota[A])$ the supremum of $A$. 

For any set $S$, let $\Func_{<\lambda}(S,\kappa)$ denote the set of partial functions $f\colon S\rightarrow\kappa$ of size ${<}\lambda$, partially ordered by reverse inclusion. 
For any ordinal $\alpha$, we write $2^{(\alpha)}$ for $\Fun(\alpha,2)$. 
Let  $\CC^\alpha:=\Func_{<\omega}(\alpha\times\omega,2)$ and $\CC:=\CC^1$. 
$\CC^\alpha$ is isomorphic to the product of $\alpha$ many copies of Cohen forcing $\CC$ with finite support. 
For any cardinal $\kappa$, a forcing is called \emph{$\kappa$-c.c.} if it does not contain antichains of size $\kappa$.

\subsection{Iterated forcing} 

A $2$-step iteration has to be restricted to names of bounded ranks to avoid the use of proper classes, and the same should happen uniformly at every stage of an iterated forcing. 
For instance, one can restrict the names for elements of the next forcing to the least possible $V_\alpha$ in each step of the iteration. 
The next lemma gives a clear account of this by providing names with optimal ranks. 
To state the lemma, we define the \emph{$\PP$-rank} $\rank_\PP(\tau):=\sup\{ \rank_\PP(\sigma)+1 \mid \exists p\ (\sigma,p)\in \tau \} $ of a $\PP$-name $\tau$ by induction on the rank of $\tau$.\footnote{Then $\rank(\sigma) \leq  \rank(\PP)+ 3\cdot \rank_\PP(\sigma)+1$ by induction.} 

In the proof, we will work with a generic filter over $V$ for convenience. 
In more detail, one can run the argument in a Boolean-valued model $V^\BB$, where $\BB=\BB(\PP)$. 
$V^\BB$ believes that it is of the form $V[G]$ for a $\PP$-generic filter $G$ over $V$. 
Every statement claimed in the proof holds in $V^\BB$ with Boolean value $\one_\BB$.\footnote{The usual argument in $\ZFC$ for working with generic filters in $V$ uses countable elementary submodels of some $H_\theta$, but such submodels need not exist in models of $\ZF$.} 

\begin{lemma} 
\label{small names} 
There is a formula defining a class function $F\colon V^3 \rightarrow V$ such that the following 
hold for any forcing $\PP$ 
and any $\PP$-name $\tau$: 
\begin{enumerate-(1)} 
\item 
\label{small names 1}
$F(\PP,q,\tau)$ is a $\PP$-name with 
$q \Vdash_\PP F(\PP,q,\tau)=\tau$. 

\item 
\label{small names 2}
If $\beta\in\Ord$ and $q\Vdash_\PP \rank(\tau)\leq\beta$, 
then 
$\rank_\PP(F(\PP,q,\tau)) \leq \beta$. 

\end{enumerate-(1)} 
\end{lemma} 
\begin{proof} 
We define $F$ by induction on the $\PP$-rank of $\tau$ by letting 
$$F(\PP,q,\tau)=\{(\nu, s)\mid \exists \sigma,r\ (\sigma,r)\in\tau,\ \nu=F(\PP,s,\sigma),\ s\leq  q,r,\ \exists \gamma\ s\Vdash_\PP \rank(\sigma)=\gamma \}.$$ 

\ref{small names 1}: 
It suffices to show $\tau^G = F(\PP,q,\tau)^G$ for any $\PP$-generic filter over $V$ with $q\in G$ by the forcing theorem. 
First, suppose $x\in \tau^G$. 
Then $x=\sigma^G$ for some $(\sigma,r)\in \tau$ with $r\in G$. 
There exist $s\leq q,r$ in $G$ and $\gamma$ with $s\Vdash_\PP \rank(\sigma)=\gamma$ by the forcing theorem. 
Then $x=\sigma^G=F(\PP,s,\sigma)^G\in F(\PP,q,\tau)^G$ by the inductive hypothesis \ref{small names 1}. 
Conversely, suppose $x\in F(\PP,q,\tau)^G$. 
There exist $(\sigma,r)\in\tau$ and $s\leq  q,r$ in $G$ with $x=F(\PP,s,\sigma)^G$. 
Then $x=F(\PP,s,\sigma)^G=\sigma^G\in \tau^G$ by the inductive hypothesis \ref{small names 1}. 

\ref{small names 2}: 
We prove the claim by induction on $\beta$. 
Suppose that $q\Vdash_\PP \rank(\tau)\leq\beta$. 
Take an element $(\nu, s)$ of $F(\PP,q,\tau)$ and witnesses $\sigma$, $r$, $\gamma$ as in the definition of $F(\PP,q,\tau)$. 
It suffices to show $\rank_\PP(\nu)<\beta$. 
Since $s\Vdash \rank(\sigma)=\gamma$ and $s\Vdash_\PP \sigma\in\tau $, we have $\gamma<\beta$. 
Since $\nu=F(\PP,s,\sigma)$, we have $\rank_\PP(\nu)=\rank_\PP(F(\PP,s,\sigma))\leq\gamma$ by the  inductive hypothesis \ref{small names 2}. 
\end{proof} 

We now fix a definition of $2$-step iterations. 

\begin{definition} 
Suppose that $\PP$ is a forcing 
and $\dot{\QQ}$ is a $\PP$-name for a forcing with $\one\Vdash_\PP \rank(\dot{\QQ})\leq \beta+1$. 
Let $\PP*\dot{\QQ}$ denote the 
set of all pairs $(p,\dot{q})$ with $p\Vdash \dot{q}\in \dot{\QQ}$, where $p\in\PP$ and $\dot{q}$ is a $\PP$-name with $\rank_\PP(\dot{q})\leq\beta$. 
\end{definition} 

It follows from Lemma \ref{small names} that the $(\PP*\dot{\QQ})$-generic extensions of $V$ are precisely those of the form $V[G][H]$, where $V[G]$ is a $\PP$-generic extension of $V$ and $V[G][H]$ is a $\dot{\QQ}^G$-generic extension of $V[G]$. 

We define iterated forcing 
by iterating the above definition of $2$-step iterations as in \cite[Section 7]{cummings2010iterated}. 
Every iteration has finite support.\footnote{We do not study iterations with other supports here. One can for example study bounded support iterations of length $\omega_1$, as mentioned after Theorem \ref{char collapse} below.}
We shall write 
$$\vec{\PP}=\langle \langle \PP_\alpha, \leq_\alpha, \one_{\PP_\alpha}\rangle, \langle \dot{\PP}_\alpha, \dot{\leq}_\alpha, \dot{\one}_{\dot{\PP}_\alpha}\rangle, \langle \PP_\gamma, \leq_\gamma, \one_{\PP_\gamma}\rangle \mid \alpha< \gamma\rangle$$ 
for an iteration of length $\gamma$.\footnote{This is standard terminology from \cite[Section 7]{cummings2010iterated} with the tweak of introducing $\dot{\one}_{\dot{\PP}_\alpha}$ due to the absence of the axiom of choice and a minor notational difference. 
Note that a $2$-step iteration has length $1$, since it consists of $\PP_0$ and $\PP_1:=\PP_0 *\dot{\PP}_1$.} 
We will abbreviate this by writing 
$\vec{\PP}=\langle \PP_\alpha,\dot{\PP}_\alpha, \PP_\gamma\mid \alpha< \gamma\rangle$ and sometimes just $\PP_\gamma$. 
For any $\PP_\gamma$-generic filter $G$ over $V$, let $G{\upharpoonright}\alpha:=\{ p{\upharpoonright}\alpha \mid p\in G \}$ for $\alpha\leq\gamma$. Also let $G_\alpha$ denote the filter induced by $G$ on $\dot{\PP}_\alpha^{G{\upharpoonright}\alpha}$ for $\alpha<\gamma$ and let $\dot{G}_\alpha$ be the canonical $\PP_\alpha$-name for $G_\alpha$.  
The \emph{support} $\supp(p)$ of a condition $p\in \PP_\gamma$ is the set of $\alpha<\kappa$ with 
$p(\alpha)\neq \dot{\one}_{\dot{\PP}_\alpha}$. 

\begin{remark} 
\label{small names for classes} 
An iterated forcing is often given by a definition for the forcing at step $\alpha$ in the $\PP_\alpha$-generic extension. 
Without the axiom of choice, we have to provide names for such forcings in a uniform fashion. 
The following suffices. 
If $\varphi$ is a formula with two free variables and $x$ is a set, let 
$A_{\varphi,y}$ 
denote the class $\{x\mid \varphi(x,y)\}$ if this class is a set and $\emptyset$ if it is a proper class. 
We claim that for any formula $\varphi$ with two free variables, there is a formula defining a class function $G\colon V^2 \rightarrow V$ such that in any generic extension of $V$, for any forcing $\PP$ and any $\PP$-name $\sigma$, 
$G(\PP,\sigma)$ is a $\PP$-name with 
$\one \Vdash_\PP G(\PP,\sigma)=A_{\varphi ,\sigma}$. 
%
%
To see this, let $\gamma$ be least with $\one\Vdash_\PP \rank(A_{\varphi,\sigma})\leq\gamma$ and define 
$$G(\PP,\sigma):=\{(\nu, q)\mid \exists \beta<\gamma\ \rank_\PP(\nu)\leq\beta,\ q\Vdash_\PP \nu \in A_{\varphi,\sigma} \}. $$
It suffices to show $\mathbf{1}\Vdash_\PP A_{\varphi,\sigma} \subseteq G(\PP,\sigma)$. 
To see this, suppose that $p\Vdash \mu\in A_{\varphi,\sigma}$. 
Then there is some $q\leq p$ and some $\beta<\gamma$ with $q\Vdash \rank(\mu)=\beta$. 
By Lemma \ref{small names}, $\nu:=F(\PP,q,\mu)$ satisfies $\rank_\PP(\nu)\leq\beta$ 
and $q\Vdash \mu= \nu$. 
Then $(\nu,q)\in G(\PP,\sigma)$ and thus $q\Vdash_\PP \mu \in G(\PP,\sigma)$. 
\end{remark}

\section{Cardinal preserving forcings} 
\label{section card pres} 

In this section, we provide a toolbox for proving that a forcing preserves cardinals. 
The main tool is a condition that strengthens the countable chain condition and implies that all cardinals and cofinalities are preserved.

\subsection{Narrow forcings} 
\label{section narrow forcing}

The following describes a variant of the countable chain condition that we call \emph{narrow}. 
In models of the axiom of choice, the two are equivalent. 
We show that all $\sigma$-linked forcings are narrow at the beginning of Section \ref{subsec:linked forcing} 
and that all random algebras are narrow in Theorem \ref{random complete}. 
We equip $\Ord$ with the discrete partial order where no two distinct elements are comparable.

\begin{definition}
\label{def narrow} 
Suppose that $\PP$ is a forcing and $\theta$, $\nu$ are ordinals, where $\theta$ is infinite. 
\begin{enumerate-(1)} 
\item 
$\PP$ is called \emph{$(\theta,\nu)$-narrow} if for any ordinal $\mu\leq\nu$ and any sequence $\vec{f}=\langle f_i \mid i<\mu \rangle$ of partial $\parallel$-homomorphisms $f_i\colon \PP \rightharpoonup \Ord$,\footnote{Thanks to Asaf Karagila for his observation in December 2022 that $\parallel$-homomorphisms $f\colon \PP \rightharpoonup \Ord$ correspond to wellordered antichains in the Boolean completion of $\PP$. 
Thus $\theta$-narrow is equivalent to the condition that the Boolean completion is $\theta^+$-c.c. 
Note that one could translate the following proofs about narrow forcing and its variants to Boolean completions.} 
$$|\bigcup_{i<\mu} \ran(f_i)|\leq |{\max(\theta,\mu)}|.$$ 

\item 
\label{def narrow 2} 
$\PP$ is called \emph{$\theta$-narrow} if it is $(\theta,\nu)$-narrow for all $\nu\in\Ord$. 
It is called \emph{narrow} if it is $\omega$-narrow. 
\end{enumerate-(1)} 
\end{definition} 

We further call $\PP$ \emph{uniformly $\theta$-narrow} if there exists a function $G$ that sends each partial $\parallel$-homomorphism $f\colon \PP\rightharpoonup \Ord$\footnote{We can assume that $\ran(f)$ is an ordinal.} 
to an injective function $G (f)\colon \ran(f) \rightarrow \theta$. 
It is called \emph{uniformly narrow} if it is uniformly $\omega$-narrow.

It is easy to see that $(\theta,\theta)$-narrow implies $(\theta,\nu)$-narrow for all $\nu$, since the condition in the definition is already satisfied for all $\mu\geq \theta^+$ if the forcing is $(\theta,1)$-narrow. 
Moreover, any uniformly $\theta$-narrow forcing is $\theta$-narrow. 
Note that $(\theta,1)$-narrow is a variant of the $\theta^+$-c.c., 
since it implies that there are no antichains of size $\theta^+$. 
Conversely, it is easy to show that any wellordered $\theta^+$-c.c. forcing $\PP$ is $\theta$-narrow by working in $\HOD_{\PP,\vec{f}}$ for any $\vec{f}$ as above, since $\PP\cap \HOD_{\PP,\vec{f}}$ is $\nu$-c.c. in $\HOD_{\PP,\vec{f}}$ for some $\nu<\theta^+$ if $\theta^+$ is singular in $\HOD_{\PP,\vec{f}}$.\footnote{Thanks to Asaf Karagila for sending us a direct proof that wellordered c.c.c. forcings preserve cardinals in October 2022.} 

Note that if $\theta^+$ is regular, then $(\theta,1)$-narrow implies $\theta$-narrow. 
Moreover, if there exists a sequence of injective functions from all $\alpha<\theta^+$ into $\theta$, then $(\theta,1)$-narrow implies uniformly $\theta$-narrow. 
We do not know if every $(\theta,1)$-narrow forcing is $\theta$-narrow and whether every $\theta$-narrow forcing is uniformly $\theta$-narrow.\footnote{These questions were solved by Elliot Glazer after this paper was submitted, see Section \ref{section: open problems}.} 
Moreover, we do not know if $(\theta,1)$-narrow forcings preserve $\theta^+$ for all $\theta\in\Card$.\footnote{However, this holds for $\theta=\omega$ by an argument of Karagila, Schilhan and the second-listed author.}  
However, $\theta$-narrow forcings preserve all cardinals ${>}\theta$ by the next lemma.


\begin{lemma} 
\label{narrow pres card} \ 
\begin{enumerate-(1)} 
\item 
\label{narrow pres card 1} 
Every $(\theta,1)$-narrow forcing $\PP$ preserves all cardinals and cofinalities ${\geq}\theta^{++}$.\footnote{This can also be proved via a result of Karagila and Schweber \cite[Proposition 5.7]{karagila2022choiceless}} 

\item 
\label{narrow pres card 2} 
Every $\theta$-narrow forcing $\PP$ preserves all cardinals and cofinalities ${\geq}\theta^+$.

\end{enumerate-(1)} 
\end{lemma} 
\begin{proof} 
\ref{narrow pres card 1}: 
We first show that $\PP$ preserves any cardinal $\lambda\geq\theta^{++}$. 
Suppose that $\mu<\lambda$ is a cardinal, $\dot{f}$ is a $\PP$-name, and $p \Vdash_\PP \dot{f} \colon \mu \rightarrow \lambda$ for some $p\in \PP$. 
For each $\alpha<\mu$, let $D_\alpha$ denote the set of all $q\leq p$ in $\PP$ that decide $\dot{f}(\alpha)$. 
Define $f_\alpha\colon D_\alpha \rightarrow \lambda$ by sending each $q$ to the unique $\beta<\lambda$ with $q \Vdash \dot{f}(\alpha)=\beta$. 
Note that each $f_\alpha$ is a partial $\parallel$-homomorphism on $\PP$.  
Since $\PP$ is $(\theta,1)$-narrow, $\otp(\ran(f_\alpha))<\theta^+$ for each $\alpha<\mu$. 
Hence $|\bigcup_{\alpha<\mu} \ran(f_\alpha)|\leq |{\max(\theta^+,\mu)}|<\lambda$. 
Hence $p$ forces that $\dot{f}$ is not surjective. 

A similar argument works for cofinalities. 
Suppose that $\lambda$ is a cardinal with $\cof(\lambda)\geq\theta^{++}$. 
Suppose that $\mu<\cof(\lambda)$ is a cardinal, $\dot{f}$ is a $\PP$-name, and $p \Vdash_\PP \dot{f} \colon \mu \rightarrow \lambda$ for some $p\in \PP$. 
With the same notation as above, $|\bigcup_{\alpha<\mu} \ran(f_\alpha)|\leq |{\max(\theta^+,\mu)}|<\cof(\lambda)$, so $p$ forces that $\dot{f}$ is not cofinal.

\ref{narrow pres card 2}: 
We first show that $\PP$ preserves $\theta^+$. 
Suppose that $\mu<\theta^+$ is a cardinal, $\dot{f}$ is a $\PP$-name, and $p \Vdash_\PP \dot{f} \colon \mu \rightarrow \theta^+$ for some $p\in \PP$. 
For each $\alpha<\mu$, let $D_\alpha$ denote the set of all $q\leq p$ in $\PP$ that decide $\dot{f}(\alpha)$. 
Define $f_\alpha\colon D_\alpha \rightarrow \theta^+$ by sending each $q$ to the unique $\beta<\theta^+$ with $q \Vdash \dot{f}(\alpha)=\beta$. 
Note that each $f_\alpha$ is a partial $\parallel$-homomorphism on $\PP$.  
Since $\PP$ is $\theta$-narrow, we have $|\bigcup_{\alpha<\mu} \ran(f_\alpha)|\leq |{\max(\theta,\mu)}|<\theta^+$. 
Hence $p$ forces that $\dot{f}$ is not surjective. 

A similar argument works for cofinalities. 
Suppose that $\lambda$ is a cardinal with $\cof(\lambda)=\theta^+$. 
Suppose that $\mu<\cof(\lambda)$ is a cardinal, $\dot{f}$ is a $\PP$-name, and $p \Vdash_\PP \dot{f} \colon \mu \rightarrow \lambda$ for some $p\in \PP$. 
With the same notation as above, $|\bigcup_{\alpha<\mu} \ran(f_\alpha)|\leq |{\max(\theta,\mu)}|<\cof(\lambda)$, so $p$ forces that $\dot{f}$ is not cofinal. 
\end{proof}

%
%

\begin{lemma} 
\label{perp hom narrow} 
Suppose that $\theta$, $\nu$ are cardinals, where $\theta$ is infinite, and $f\colon \PP\rightarrow \QQ$ is a $\perp$-homomorphism. 
\begin{enumerate-(1)} 
\item 
\label{perp hom narrow 1} 
$\QQ$ is $(\theta,\nu)$-narrow, then $\PP$ is $(\theta,\nu)$-narrow. 

\item 
\label{perp hom narrow 2} 
$\QQ$ is uniformly $\theta$-narrow, then $\PP$ is uniformly $\theta$-narrow. 
\end{enumerate-(1)} 
\end{lemma} 
\begin{proof} 
\ref{perp hom narrow 1}
Suppose that 
$\vec{f}=\langle f_i \mid i<\mu \rangle$ is a sequence of partial $\parallel$-homomorphisms $f_i\colon \PP \rightharpoonup \Ord$. 
Let $D:=\ran(f)$ and define $g_i\colon D\rightarrow \Ord$ as follows. 
Note that for all $p,r\in \PP$ with $f(p)=f(r)$, we have $f_i(p)=f_i(r)$, since $f$ is a $\perp$-homomorphism and $f_i$ is a $\parallel$-homomorphism. 
For $f(p)=q\in D$, we can thus define $g_i(q)=f_i(p)$. 
We claim that each $g_i$ is a partial $\parallel$-homomorphism. 
Suppose that $q,s\in D$ with $f(p)=q$, $f(r)=s$ and $q\parallel s$. 
Since $f$ is a $\perp$-homomorphism, $p\parallel r$. 
Since $f_i$ is a $\parallel$-homomorphism, $g_i(q)=f_i(p)\parallel f_i(r)=g_i(s)$ as desired. 
Since $\ran(f_i)=\ran(g_i)$ for all $i<\mu$ and $\QQ$ is $(\theta,\nu)$-narrow, the statement of the lemma follows. 

\ref{perp hom narrow 2} 
Suppose $G$ witnesses that $\QQ$ is uniformly $\theta$-narrow. 
The proof of \ref{perp hom narrow 1} defines a function $H$ from $G$ that witnesses $\PP$ is uniformly $\theta$-narrow by mapping a partial $\parallel$-homomorphism $f\colon \PP\rightharpoonup \Ord$ 
to a partial $\parallel$-homomorphisms $g$ on $\QQ\rightharpoonup \Ord$ with $\ran(f)=\ran(g)$. 
\end{proof}




We want to avoid collapsing cardinals when iterating $\theta$-narrow forcings. 
However, precisely this will happen if we are not careful. 
To see this, recall that both $\theta$-narrow and uniformly $\theta$-narrow are equivalent to the $\theta^+$-c.c. for wellordered forcings. 
Using this, we argue that an iteration of narrow forcings with finite support can collapse cardinals. 
Suppose that $\omega_1$ is singular and $\vec{\alpha}=\langle \alpha_n \mid n\in\omega \rangle$ is cofinal in $\omega_1$. 
Suppose that $\PP_n$ is the discrete partial order on $\alpha_n$. 
The finite support iteration of the forcings $\PP_n$ has a dense subset isomorphic to the finite support product $\prod_{n\in\omega} \PP_n$. 
It is easy to see that this collapses $\omega_1$. 
We therefore need to take precautions. 
Suppose that $\theta$ is an infinite ordinal. 
A \emph{uniform iteration} of $\theta$-narrow forcings with finite support is a sequence 
$\vec{\PP}=\langle \PP_\alpha,\dot{\PP}_\alpha,\dot{G}_\alpha,\PP_\gamma\mid \alpha< \gamma\rangle$ 
such that 
$\vec{\PP}=\langle \PP_\alpha,\dot{\PP}_\alpha,\PP_\gamma\mid \alpha< \gamma\rangle$ is a finite support iteration and for each $\alpha<\gamma$, $\one_{\PP_\alpha}$ forces that $\dot{\PP}_\alpha$ is uniformly $\theta$-narrow witnessed by $\dot{G}_\alpha$. 


\begin{theorem} 
\label{iteration theorem} 
Suppose that $\theta$ is an infinite ordinal. 
Any uniform iteration of $\theta$-narrow forcings with finite support is again uniformly $\theta$-narrow. 
\end{theorem} 
\begin{proof} 
We can assume $\theta\in \Card$. 
Suppose that $\vec{\PP}=\langle \PP_\alpha,\dot{\PP}_\alpha, \dot{g}_\alpha, \PP_\delta \mid \alpha<\delta\rangle$ 
is such an iteration. 
We construct a sequence $\vec{G}=\langle G_\gamma \mid \gamma\leq\delta\rangle$ of functions  by induction on $\gamma\leq\delta$, where $G_\gamma$ witnesses that $\PP_\gamma$ is uniformly $\theta$-narrow. 
$\vec{G}$ will be defined from $\vec{\PP}$ and $\theta$ by recursion. 
We will assume that $\ran (f)$ is an ordinal for any input $f$ of $G_{\gamma}$ 
to ensure that $G_\gamma$ is a set function. 

First suppose $\gamma=\beta+1$. 
The following provides a construction of $G_{\beta+1}$ from $G_\beta$ that works uniformly for all $\beta<\delta$. 
Fix a partial $\parallel$-homomorphism $f\colon \PP_\beta*\dot{\PP}_\beta\rightharpoonup \Ord$. 
For $\PP_{\beta}$-names $\sigma$ and $\tau$, let $\pair(\sigma,\tau)$ denote the canonical $\PP_{\beta}$-name for the ordered pair $(\sigma,\tau)$. 
Let 
$$ \dot{f} := \{ (\pair(\dot{q},\check{\alpha}), p) \mid f(p,\dot{q})=\alpha \}. $$ 

\begin{claim*} 
$\one_{\PP_{\beta}}$ forces that $\dot{f}$ is a partial $\parallel$-homomorphism on $\dot{\PP}_\beta$. 
\end{claim*} 
\begin{proof} 
Suppose that $G$ is $\PP_{\beta}$-generic over $V$. 
In $V[G]$, take $q_0, q_1\in \dot{\PP}_{\beta}^G$ with $\dot{f}^G (q_i)=\alpha_i$ for $i<2$ and $\alpha_0\neq\alpha_1$. 
We claim $q_0\perp q_1$. 
There exist $\dot{q}_i$ with $\dot{q}_i^G=q_i$ and $p_i \in G$ with $(\mathrm{pair}(\dot{q}_i,\check{\alpha}_i), p_i)\in \dot{f}$   for $i<2$. 
If $q_0$ and $q_1$ were compatible, then some $p\in G$ forces that $\dot{q}_0$ and $\dot{q}_1$ are compatible.  
Since we can assume $p\leq p_0,p_1$, then $(p_0,\dot{q}_0)$ and $(p_1,\dot{q}_1)$ would be compatible. 
But $f(p_0,\dot{q}_0)=\alpha_0 \neq \alpha_1=f(p_1,\dot{q}_1)$ and $f$ is a $\parallel$-homomorphism. 
\end{proof} 

By the previous claim, $\one_{\PP_\beta}$ forces that $\dot{g}_\beta(\dot{f})$ is an injective
function from $\ran(\dot{f})$ into $\theta$. 
This induces a $\PP_\beta$-name $\dot{h}$ for a surjection from $\theta$ onto $\ran(\dot{f})$.\footnote{The inverse of $\dot{g}_\beta(\dot{f})$ is partial function from $\theta$ onto $\ran(\dot{f})$. 
One can take $\dot{h}$ to be a nice name for a total function extending it.} 
For each $\alpha<\theta$, let $D_\alpha$ denote the set of all $p\in \PP_\beta$ that decide $\dot{h} (\alpha)$. 
For each $\alpha<\theta$, define a $\parallel$-homomorphism $h_\alpha\colon D_\alpha\rightarrow \Ord$ by letting $h_\alpha(p)$ be the unique $\delta$ such that $p$ forces $\dot{h}(\alpha)=\delta$. 
Since $G_{\beta}$ witnesses that $\PP_{\beta}$ is uniformly $\theta$-narrow, $G_{\beta} (h_{\alpha}) \colon \ran (h_{\alpha}) \to \theta$ is injective for each $\alpha < \theta$, and the sequence $(G_{\beta}(h_{\alpha}) \mid \alpha < \theta)$ induces an injection from $\bigcup_{\alpha<\theta} \ran(h_\alpha)$ to $\theta$.
Since $\one_\PP$ forces $\ran(\dot{f})\subseteq \bigcup_{\alpha<\theta} \ran(h_\alpha)$, we  have $\ran(f)\subseteq \bigcup_{\alpha<\theta} \ran(h_\alpha)$ by the definition of $\dot{f}$. 
Thus 
$G_{\beta}$ induces an injective map $G_\gamma(f)\colon \ran(f) \rightarrow \theta$ uniformly in $f$, and $G_{\gamma}$ is definable from $\vec{\PP}$, $G_\beta$, and ordinals. 




Now suppose $\gamma$ is a limit ordinal. 
Suppose that $f\colon \PP_\gamma\rightharpoonup \Ord$ is a partial $\parallel$-homomorphism. 
\begin{claim*} 
$|\ran(f)|^{\HOD_{\vec{\PP},f}}\leq \theta$. 
\end{claim*} 
\begin{proof} 
Fix a wellorder $\leq^*$ of $[\Ord]^{<\omega}$ that is definable without parameters. 
Define $\vec{s}=\langle s_\alpha \mid \alpha \in \ran(f)\rangle$ as follows. 
For each $\alpha \in \ran(f)$, let $s_\alpha$ be the $\leq^*$-least element $s$ of $[\gamma]^{<\omega}$ such that there exists some $p\in \PP_\gamma$ with support $s$ and $f(p)=\alpha$. 
We can assume all $p$ with $f(p)=\alpha$ have support $s_\alpha$ for each $\alpha\in \ran(f)$ by restricting $f$ while keeping its range intact. 

Towards a contradiction, suppose the claim fails for $f$. 
We can assume $|\ran(f)|^{\HOD_{\vec{\PP},f}}=\theta^{+\HOD_{\vec{\PP},f}}$.
By applying the infinite sunflower lemma\footnote{
We apply Lemma \cite[Theorem 1.6]{Kunen}, also known as the \emph{$\Delta$-system lemma}, to $\theta^{+\HOD_{\vec{\PP},f}}$. 
Recall that a \emph{sunflower} or \emph{$\Delta$-system} is a collection $S$ of sets such that the intersection $a\cap b=c$, the \emph{centre}, is the same for all $a,b\in S$ with $a \neq b$.}  in $\HOD_{\vec{\PP},f}$ and restricting $f$, we can assume $\vec{s}$ is a sunflower system with centre $r$. 
Let $\beta=\max(r)+1<\gamma$ if $r\neq\emptyset$ and $\beta=0$ otherwise. 
Let $A=\{ p{\upharpoonright}\beta \mid p\in \dom(f) \}$ denote the projection of $\dom(f)$ to $\beta$. 
Define $g\colon A\rightarrow \Ord$ as follows. 
For each $p\in A$, let 
$$ g(p)=\alpha :\Longleftrightarrow  \exists s\in \dom(f)\ \ (s{\upharpoonright}\beta=p \wedge f(s)=\alpha). $$ 
To see that $g$ is well-defined, 
note that for any $p\in A$ and $t,u\in \dom(f)$ with $t{\upharpoonright}\beta=u{\upharpoonright}\beta=p$, we have $f(t)=f(u)$. 
Otherwise let $f(t)=\zeta$ and $f(u)=\xi$ with $\zeta\neq\xi$. 
Since $\supp(t)=s_\zeta$, $\supp(u)=s_\xi$, $s_\zeta\cap s_\xi= \supp (p)$ and $\vec{s}$ is a sunflower system with centre $r$, $t$ and $u$ are compatible. 
But $f$ is a $\parallel$-homomorphism and $f(t)\neq f(u)$. 

\begin{subclaim*} 
$g$ is a $\parallel$-homomorphism on $\PP_\beta$ with $\ran(f)=\ran(g)$. 
\end{subclaim*} 
\begin{proof} 
We first show that $g$ is a $\parallel$-homomorphism. 
Suppose that $p,q \in A$ with $g(p)\neq g(q)$. 
By the definition of $g$, there exist $t, u\in \dom(f)$ with $t{\upharpoonright}\beta=p$, $\supp(t)=s_\zeta$, $u{\upharpoonright}\beta=q$, $\supp(u)=s_\xi$ and $\zeta\neq \xi$. 
Since $f(t)\neq f(u)$ and $f$ is a $\parallel$-homomorphism, $t$ and $u$ are incompatible. 
Since $\vec{s}$ is a sunflower with centre $r$, this implies $p=t{\upharpoonright}\beta$ and $q=u{\upharpoonright}\beta$ are incompatible. 

To see that $\ran(f)=\ran(g)$, take any $\alpha \in \ran(f)$. 
Pick any $p$ with support $s_\alpha$ and $f(p)=\alpha$. 
Then $p{\upharpoonright}\beta\in A$ and $g(p{\upharpoonright}\beta)=\alpha$ by the definition of $A$ and $g$. 
\end{proof} 

By the inductive hypothesis, $G_\beta$ witnesses that $\PP_\beta$ is uniformly $\theta$-narrow. 
The previous subclaim yields an injective function $G_\beta(g)\colon \ran(g)\rightarrow \theta$. 
Since $G_\beta$ and $g$ are definable from $\vec{\PP}$, $f$, and ordinals, 
we have $|\ran(f)|^{\HOD_{\vec{\PP},f}}=|\ran(g)|^{\HOD_{\vec{\PP},f}}\leq \theta$, contradicting the assumption. 
\end{proof} 

Using the previous claim, let $G_{\gamma} (f)$ be the least injective function $h \colon \ran(f) \rightarrow \theta$ in $\HOD_{\vec{\PP},f}$. 
Then the definition of $G_{\gamma} (f)$ is uniform in $f$ and $G_{\gamma}$ is definable from $\vec{\PP}$ and ordinals. 
\end{proof}

\subsection{Linked forcings}\label{subsec:linked forcing} 

A forcing $\PP$ is called \emph{$\QQ$-linked} if there is a $\perp$-homomorphism from $\PP$ to $\QQ$. 
We equip each ordinal $\theta$ with the discrete partial order. 
Clearly $\theta$ is uniformly $\theta$-narrow. 
Every $\theta$-linked forcing is uniformly $\theta$-narrow by Lemma \ref{perp hom narrow}. 

Suppose that $\theta$ is an ordinal. 
A \emph{uniform iteration} of $\theta$-linked forcings with finite support is a sequence 
$\vec{\PP}=\langle \PP_\alpha,\dot{\PP}_\alpha,\dot{f}_\alpha,\PP_\gamma\mid \alpha< \gamma\rangle$ 
such that 
$\vec{\PP}=\langle \PP_\alpha,\dot{\PP}_\alpha,\PP_\gamma\mid \alpha< \gamma\rangle$ is a finite support iteration and for each $\alpha<\gamma$, $\one_{\PP_\alpha}$ forces 
that $\dot{\PP}_\alpha$ is $\theta$-linked witnessed by $\dot{f}_\alpha$. 
Uniform products of $\theta$-linked forcings with finite support are defined similarly. 
Given a uniform iteration of $\theta$-linked forcings, the proof of Lemma \ref{perp hom narrow} provides a map that turns a function witnessing $\theta$-linked into a function witnessing $\theta$-narrow. 
This supplies us with a sequence $\vec{G}=\langle \dot{G}_\alpha \mid \alpha<\gamma\rangle$ such that $\vec{\PP}=\langle \PP_\alpha,\dot{\PP}_\alpha,\dot{G}_\alpha,\PP_\gamma\mid \alpha< \gamma\rangle$ is a uniform iteration of $\theta$-narrow forcings. 
Note that any uniform product of $\theta$-linked forcings with finite support can also be understood as a uniform iteration of $\theta$-linked forcings. 
Lemma \ref{narrow pres card} and Theorem~\ref{iteration theorem} then have the following corollary. 

\begin{corollary} 
\label{iteration linked} 
Uniform iterations and products of $\theta$-linked forcings with finite support preserve all cardinals and cofinalities ${>}\theta$. 
\end{corollary} 

For instance, this works for mixed finite support products and iterations in any order of Cohen forcing, Hechler forcing and eventually different forcing. 
These forcings are $\omega$-linked. 
To see that products and iterations are uniform, it suffices by Remark \ref{small names for classes} that there is a definition for a linking function that works uniformly in all models. 

\begin{remark} 
One can obtain a direct proof that any uniform iteration $\vec{\PP}$ of $\theta$-linked forcings preserves cardinals and cofinalities ${>}\theta$ without going through $\theta$-narrow. 
Using the above notation, let $\QQ$ denote the set of $p\in \PP_\gamma$ such that for all $\alpha\in \supp(p)$, $p{\upharpoonright}\alpha$ decides $\dot{f}_\alpha(p(\alpha))$. 
We will show in Lemma \ref{density of conditions deciding support} below that $\QQ$ is dense in $\PP_\gamma$. 
One then constructs a $\perp$-homomorphism from $\QQ$ to 
$\Fun(\gamma,\theta)$. 
This can be extended to a $\perp$-homomorphism on $\PP_\gamma$. 
It is now easy to show directly that the existence of a $\perp$-homomorphism from $\PP$ to a wellordered $\theta^+$-c.c. forcing implies that $\PP$ preserves cardinals ${>}\theta$. 
Alternatively we can argue that $\Fun(\gamma,\theta)$ is $\theta$-narrow, since it is wellordered and 
$\theta^+$-c.c., so $\PP_\gamma$ is $\theta$-narrow by Lemma \ref{perp hom narrow} \ref{perp hom narrow 2} and it thus preserves cardinals and cofinalities ${>}\theta$ by Lemma \ref{narrow pres card}. 
\end{remark} 

\subsection{Locally complete forcings} 

We present criteria for showing that a forcing is $\theta$-narrow or uniformly $\theta$-narrow that may be of independent interest. 
This is useful to show that random algebras are uniformly narrow and thus preserve all cardinals and cofinalities. 
The notions introduced here are also used to study absoluteness principles in Section \ref{generic absoluteness}. 
The next definition is motivated by the property of random forcing that inner models are correct about predense sets.



\begin{lemma} 
\label{lc equiv} 
The following conditions are equivalent for a forcing $\PP$ and a finite set $x$ that contains $\PP$: 
\begin{enumerate-(a)} 
\item 
\label{lc equiv 1} 
For any regular open $A\subseteq \PP$,\footnote{A subset $A$ of $\PP$ is called \emph{regular open} if 
$A=\{ p \mid p_\iota\leq \sup(A)\}$.}    
$\sup(A\cap \HOD_{x\cup \{ A\}})=\sup(A)$. 

\item 
\label{lc equiv 2} 
For any $B\subseteq \PP$ with $B\neq\emptyset$, 
there is some $p\in \PP\cap \HOD_{x\cup \{ B\}}$ with $p_\iota\leq \sup(B)$.

\item 
\label{lc equiv 3} 
The same as \ref{lc equiv 2}, but only for regular open subsets $B$ of $\PP$. 

\end{enumerate-(a)} 
\end{lemma} 
\begin{proof} 
\ref{lc equiv 1}$\Rightarrow$\ref{lc equiv 2}: 
Let $A:=\{p \mid p_\iota\leq \sup(B)\}\neq\emptyset$. 
Since $A$ is regular open, $\sup(A\cap \HOD_{x\cup \{ A\}})=\sup(A)\neq\mathbf{0}_{\BB (\PP)}$ by \ref{lc equiv 1}. 
Thus $A\cap \HOD_{x\cup \{ A\}}\subseteq A\cap \HOD_{x\cup \{ B\}}$ is nonempty. 
The inclusion holds since $\PP\in x$. 
Any element of $A\cap \HOD_{x\cup \{ A\}}$ witnesses \ref{lc equiv 2}. 

\ref{lc equiv 2}$\Rightarrow$\ref{lc equiv 3} is obvious. 

\ref{lc equiv 3}$\Rightarrow$\ref{lc equiv 1}: 
We can assume that $A$ is nonempty. 
Towards a contradiction, suppose $a:=\sup(A)>b:=\sup(A\cap \HOD_{x\cup \{ A\}})$ in $\BB(\PP)$. 
$B:=\{ p\in \PP \mid p_\iota\leq a -b\}\neq\emptyset$ is regular open. 
Since $A$ is regular open, we have $B\subseteq A$. 
There is some $p\in \PP\cap \HOD_{x\cup\{B\}}\subseteq \HOD_{x\cup\{A\}}$ with $p_\iota\leq \sup(B)=a-b$ by \ref{lc equiv 3}. 
The inclusion holds since $\PP\in x$. 
Since $B$ is regular open, $p\in B\subseteq A$. 
In particular, $p\in A \cap \HOD_{x\cup\{A\}}$. 
However, $p_\iota\leq a-b$ is incompatible with all elements of $A\cap \HOD_{x\cup \{ A\}}$. 
\end{proof} 

Note that these conditions also hold for any finite superset of $x$. 
Any wellorderable forcing $\PP$ satisfies them, since one can pick any finite set $x$ that contains a wellorder of $\PP$ and thus $\PP\in \HOD_x$. 
Note that \ref{lc equiv 2} is a variant of the existence of suprema. 
A useful way of proving \ref{lc equiv 2} is to find some $p\in \PP\cap \HOD_x$ that is equivalent to $\sup(B)$. 
We shall do this for random algebras in Section \ref{section random complete}. 

\begin{definition} 
A forcing $\PP$ is called \emph{locally complete} if there exists a finite set $x$ containing $\PP$ such that each of the equivalent conditions \ref{lc equiv 1}-\ref{lc equiv 3} in Lemma \ref{lc equiv} holds. 
\end{definition} 

We now show that for any locally complete forcing $\PP$ and $x$ as above, any $\PP$-generic filter over $V$ is $(\PP\cap \HOD_x)$-generic over $\HOD_x$. 

\begin{lemma} 
\label{cons of lc} 
Suppose that $\PP$ is locally complete witnessed by $x$ and $y\supseteq x$ is finite. 
\begin{enumerate-(1)} 
\item 
\label{cons of lc 1} 
$\HOD_y$ is correct about compatibility of elements of $\PP\cap \HOD_y$. 
\item 
\label{cons of lc 2} 
Every predense subset $A\in \HOD_y$ of $\PP\cap \HOD_y$ is predense in $\PP$. 
\end{enumerate-(1)} 
\end{lemma} 
\begin{proof} 
\ref{cons of lc 1} 
Suppose that $p,q\in \PP\cap \HOD_y$ with $p \parallel q$. 
Then $A:=\{ r\in \PP \mid r \leq p, q \}\neq\emptyset$ is regular open. 
By \ref{lc equiv 3}, there is some $r \in \HOD_{x\cup\{A\}}\subseteq \HOD_y$ with $r_\iota\leq \sup(A)\leq p_\iota,q_\iota$. 
$r$ witnesses that $p$ and $q$ are compatible in $\HOD_y$. 

\ref{cons of lc 2} 
Towards a contradiction, suppose that $A\in \HOD_y$ is predense in $\PP\cap \HOD_y$, but not predense in $\PP$. 
Then $\sup(A)\neq \one_{\BB(\PP)}$. 
Then $B:=\{ p\in \PP \mid p_\iota\leq - \sup(A)\}\neq\emptyset$ is regular open. 
By \ref{lc equiv 3}, there is some $p\in \HOD_{x\cup\{B\}}\subseteq\HOD_y$ with $p_\iota\leq \sup(B)=-\sup(A)$. 
This contradicts the assumption that $A$ is predense in $\PP\cap \HOD_y$. 
\end{proof}

We now study locally complete $\theta^+$-c.c. forcings and a stronger variant. 
This is motivated by the fact that random forcing is c.c.c. in any inner model $\ZFC$. 
We say that a property holds for \emph{sufficiently large finite} $y$ if there exists a finite set $x$ such that it holds for all finite sets $y\supseteq x$. 
Note that a forcing $\PP$ is $\theta^+$-c.c. if and only if for sufficiently large finite $x$, 
any antichain $A\in \HOD_x$ in $\PP$ satisfies $|A|^{\HOD_x}< \theta^+$. 
The following defines a variant where the successor of $\theta$ is calculated in $\HOD_x$, 
as indicated by the notation $\theta^{[+]}$. 


\begin{definition} 
\label{def lc} 
For any infinite ordinal $\theta$, a forcing $\PP$ is called \emph{locally $\theta^{[+]}$-c.c.} if for sufficiently large finite $x$, any antichain $A\in \HOD_x$ in $\PP$ satisfies $|A|^{\HOD_x}\leq \theta$. 
\end{definition} 




Together with local completeness, this notion suffices to show that a forcing is uniformly $\theta$-narrow. 
This will allow us to show, for example, that random algebras can be iterated without collapsing cardinals. 

\begin{lemma} 
\label{cc lc imply narrow} 
Suppose $\theta$ is an infinite ordinal and $\PP$ is a locally complete forcing.  
\begin{enumerate-(1)} 
\item 
\label{cc lc imply narrow 1} 
If $\PP$ is $\theta^+$-c.c., then it is $\theta$-narrow. 

\item 
\label{cc lc imply narrow 2} 
If $\PP$ is locally $\theta^{[+]}$-c.c., then it is uniformly $\theta$-narrow. 
\end{enumerate-(1)} 
\end{lemma} 
\begin{proof} 
\ref{cc lc imply narrow 1} 
Suppose that $\vec{f}=\langle f_i \mid i<\mu \rangle$ is a sequence of partial $\parallel$-homomorphisms from $\PP$ to $\Ord$ for some ordinal $\mu$. 
Let $A_i:=\dom(f_i)$ for each $i<\mu$. 
We partition $A_i$ into $A_{i,\alpha}:=\{ p\in A_i \mid f_i(p)=\alpha\}$ for $\alpha\in \ran(f_i)$. 
Let $\vec{A}=\langle A_{i,\alpha}\mid i<\mu,\ \alpha \in \ran(f_i)\rangle$. 
Take some finite $x$ witnessing $\PP$ is locally complete. 
Let $H:=\HOD_{x\cup\{\vec{A}\}}$. 
For each $i<\mu$ and $\alpha\in \ran(f_i)$, there exists some $p\in \PP\cap H$ with $p\leq \sup(A_{i,\alpha})$. 
Let $p_{i,\alpha}$ be be least such $p$ in the canonical wellorder of $H$. 

\begin{claim*} 
$p_{i,\alpha}\perp p_{i,\beta}$ for any $i<\mu$ and $\alpha\neq\beta$ in $\ran(f_i)$. 
\end{claim*} 
\begin{proof} 
Towards a contradiction, suppose that there exists some $r\leq p_{i,\alpha}, p_{i,\beta}$. 
Since $r\leq \sup(A_{i,\alpha})$, there exist $s\in A_{i,\alpha}$ and $t\leq r, s$. 
Since $t\leq \sup(A_{i,\beta})$, $t$ is compatible with some $u\in A_{i,\beta}$. 
Thus $s\parallel u$. 
Since $f_i$ is a $\parallel$-homomorphism, $\alpha=f_i(s)=f_i(u)=\beta$. 
But we assumed $\alpha\neq\beta$. 
\end{proof} 

First suppose $\theta^+$ is regular in $H$. 
By the $\theta^+$-c.c., we have $|\ran(f_i)|^H<\theta^+$ for each $i<\mu$. 
Since $\theta^+$ is regular in $H$, we have $|\bigcup_{i<\mu}\ran(f_i)|^H< \theta^+$ if $\mu < \theta^+$. 
Thus  $|\bigcup_{i<\mu}\ran(f_i)|\leq |\theta|$ if $\mu < \theta^+$ and $|\bigcup_{i<\mu}\ran(f_i)|\leq |\mu|$ if $\mu \ge \theta^+$. 

Next suppose $\theta^+$ is singular in $H$. 
Let $\lambda\leq\theta^+$ be the chain condition of $\PP\cap H$ in $H$. 
By a standard fact, the chain condition of a forcing is regular in any model of $\ZFC$. 
Thus $\lambda$ is regular in $H$ and $\lambda<\theta^+$. 
Since $H$ is correct about compatiblity by Lemma \ref{cons of lc} \ref{cons of lc 1}, $|\ran(f_i)|^H<\lambda$ for each $i<\mu$. 
Thus $|\bigcup_{i<\mu}\ran(f_i)|\leq |\theta|$ if $\lambda \geq \mu$ 
and $|\bigcup_{i<\mu}\ran(f_i)|\leq |\mu|$ if $\lambda < \mu$. 

\ref{cc lc imply narrow 2} 
We proceed as in \ref{cc lc imply narrow 1}, but replace $\vec{f}$ by a single $\parallel$-homomorphism $f$ and $x$ by a finite superset that witnesses locally-$\theta^{[+]}$-c.c. 
The construction in \ref{cc lc imply narrow 1} works as above. 
Since the $\theta^{+H}$-c.c. holds in $H$ by the choice of $x$, 
we have $|\ran(f)|^H\leq|\theta|^H$. 
Let $G(f)$ be the least injective function $F\colon \ran(f) \rightarrow \theta$  in $H$. 
\end{proof} 



The following is an alternative way of showing cardinal preservation. 
It uses a notion of capturing similar to the ones studied in \cite{castiblanco2021preserving,schschsch}. 

\begin{definition} 
\label{def captured} 
Suppose $\theta$ is an infinite ordinal. 
A forcing $\PP$ is called \emph{$\theta$-captured} if for any set $A$ of ordinals in any $\PP$-generic extension $V[H]$, there are a transitive class model $U\subseteq V$ of $\ZFC$ 
and a generic filter $G\in V[H]$ over $U$ such that: 
\begin{enumerate-(a)}
\item 
\label{def captured 1} 
$A\in U[G]$. 
\item 
\label{def captured 2} 
$U$ and $U[G]$ agree on cardinals and cofinalities ${>}\theta$. 
\end{enumerate-(a)}
We further say that $\PP$ is $\theta$-captured \emph{by $\PP$} if $G=H\cap U$ is $(\PP\cap U)$-generic over $U$ in \ref{def captured 1}. 
\end{definition} 



\begin{lemma}
\label{lc capturing} 
Suppose that $\theta$ is an infinite ordinal and $\PP$ is a forcing. 
\begin{enumerate-(1)} 
\item 
\label{lc capturing 1} 
If $\PP$ is locally complete and $\theta^+$-c.c., then it is $\theta$-captured by $\PP$. 

\item 
\label{lc capturing 2} 
If $\PP$ is $\theta$-captured, then it preserves cardinals and cofinalities ${>}\theta$. 

\end{enumerate-(1)} 
\end{lemma} 
\begin{proof} 
\ref{lc capturing 1} 
Suppose that $\PP$ is locally complete witnessed by $x$. 
Recall that a \emph{nice name} for a subset of $\lambda$ has elements of the form $(\check{\alpha},p)$ with $\alpha<\lambda$ and $p\in\PP$. 
Suppose that $\sigma$ is a nice name for a subset of $\lambda$. 
Let $A_\alpha:=\{ p\in \PP \mid (\check{\alpha},p)\in \sigma\}$ for each $\alpha<\lambda$ and set $H:=\HOD_{x\cup \{\sigma\}}$. 

\begin{claim*} 
There exists a name $\dot{g}\in H$ with $\one_\PP \Vdash \dot{f}=\dot{g}$. 
\end{claim*} 
\begin{proof} 
We construct a sequence $(B_\alpha \mid \alpha < \lambda ) \in H$ of antichains with $\sup(A_\alpha)=\sup(B_\alpha)$ for each $\alpha<\lambda$. 
Then $\dot{g}=\{ (\check{\alpha},p) \mid p\in B_\alpha \}$ is as required. 
We construct $B_\alpha$ by induction. 
Let $p_0 \in H$ with $(p_0)_\iota\leq \sup(A_\alpha)$ by \ref{lc equiv 3}. 
Suppose that $\vec{p}=\langle p_i \mid i<\eta \rangle$ has been constructed. 
If $\sup(\vec{p})<\sup(A_\alpha)$, then $B:=\{ p\in \PP \mid p_\iota\leq \sup(A_\alpha)-\sup(\vec{p}) \}\neq\emptyset$ is regular open. 
Let $p_\eta \in \HOD_{x\cup\{B\}}\subseteq H$ with $(p_\eta)_\iota \leq \sup(A_\alpha)-\sup(\vec{p})$ by \ref{lc equiv 3}. 
If $\sup(\vec{p}) = \sup(A_\alpha)$, let $B_\alpha=\ran(\vec{p})$. 
Then $B_\alpha$ is as required. 
\end{proof} 

Take $\dot{g}$ as in the previous claim and suppose that $G$ is a $\PP$-generic filter over $V$. 
Then $\dot{f}^G=\dot{g}^G$. 
Moreover, $G \cap H$ is $(\PP\cap H)$-generic over $H$ by Lemma \ref{cons of lc} \ref{cons of lc 2}. 
This suffices, since $\PP\cap H$ is $\theta^+$-c.c. in $H$. 

\ref{lc capturing 2} 
Suppose that $\lambda>\theta$ is a cardinal, $\gamma<\lambda$ and $f \colon \gamma \to \lambda$ is a function in some $\PP$-generic extension $V[H]$. 
By assumption, $f$ is contained in a model $U[G]$ as above. 
Since $\lambda>\theta$ is a cardinal in $U$ and therefore in $U[G]$, $f$ is not surjective. 
A similar argument works for cofinalities. 
Suppose that $\lambda\in\Ord$ with $\cof(\lambda)>\theta$, $\gamma<\cof(\lambda)$ and $f \colon \gamma \to \lambda$ is a function in $V[H]$. 
Then $f$ is contained in a model $U[G]$ as above. 
Since $\gamma<\cof(\lambda)\leq\cof(\lambda)^U=\cof(\lambda)^{U[G]}$, $f$ is not cofinal. 
\end{proof} 

Recall that $\CC^\alpha =\Func_{<\omega}(\alpha\times\omega,2)$ is the forcing adding $\alpha$-many Cohen reals with finite conditions for any ordinal $\alpha$. Notice that $\CC^\kappa$ is wellorderable and locally $\omega^{[+]}$-c.c. for any cardinal $\kappa$. Therefore it is uniformly narrow. 
We will see in the next section that random algebras are locally complete and locally $\omega^{[+]}$-c.c.
So they are uniformly narrow as well.

\subsection{Random algebras} 
\label{section random forcing}


Fix a multiplicatively closed ordinal $\alpha$.\footnote{We make this assumption so that an $\alpha$-Borel code for a subset of $2^\alpha$ is a subset of $\alpha$.} 
$2^\alpha$ is equipped with the product topology. 
Recall that for any ordinal $\alpha$, we write $2^{(\alpha)}$ for $\Fun(\alpha,2)$. The basic open subsets of $2^\alpha$ are of the form $N_t=\{ x\in 2^\alpha \mid t\subseteq x \}$ for 
$t\in 2^{(\alpha)}$.

\subsubsection{Borel codes}  

An \emph{$\alpha$-Borel code} for a subset of $2^\alpha$ is a subset of $\alpha$ that codes a set formed from basic open subsets of $2^\alpha$ via complements and countable unions of length at most $\alpha$. 
Its \emph{rank} is the number of steps of this construction. 
We may fix any standard definition of $\alpha$-Borel codes such as the one at the end of \cite[Section 25]{Jech} where $\omega$ is replaced with $\alpha$, using the pairing function $\p\colon \Ord\times\Ord\rightarrow\Ord$. 
An \emph{$\alpha$-Borel} subset of $2^\alpha$ is one with an $\alpha$-Borel code. 
For $\alpha$-Borel codes $A_0$ and $A_1$, let $A_0\vee A_1$, $A_0\wedge A_1$, $A_0 - A_1$, $-A_0$ denote canonical codes for their union, intersection, difference and complement. 
The \emph{support} $\supp(A)$ of an $\alpha$-Borel code $A$ is the union of all $\dom(t)$, where $N_t$ occurs in $A$. 
\emph{Borel codes} are by definition $\alpha$-Borel codes with countable support. 
A \emph{Borel subset} $B$ of $2^\alpha$ is one with a Borel code. 
A \emph{support} of $B$ is by definition a support of a code for $B$. 

\begin{remark} 
\label{least support} 
Any Borel subset $B$ of $2^\alpha$ has a least support $S$. 
We define $S$ as follows. 
For any $i<\kappa$, let $i\in S$ if there exist some $x\in B$ and $y\notin B$ 
with $x(j)= y(j)$ for all $j\neq i$. 
Clearly $S$ is contained in any support of $B$. 
To see that $S$ is a support for $B$, suppose that $A$ is a Borel code for $B$. 
Define a Borel code $A{\upharpoonright} S$ by replacing $N_t$ by $N_{t{\upharpoonright} S}$ for all $t\in 2^{(\alpha)}$. 
By the definition of $S$, $A{\upharpoonright} S$ codes the same set as $A$. 
\end{remark} 

Let $\mu$ denote the product measure on $2^\alpha$. 
We sometimes identify a Borel code with the coded set. 
For instance, we write $\mu(A)$ for $\mu(B)$ if $A$ is a Borel code for the set $B$. 
For Borel codes $A_0$ and $A_1$, we write $A_0\leq A_1$ if $\mu(A_0-A_1)=0$ and $A_0=_\mu A_1$ if $A_0\leq A_1 \wedge A_0\leq A_1$. 
Note that for any $x\in 2^\alpha$ and a Borel code $A$ for a set $B$, the statement ``$x\in B$'' is absolute to any inner model of $\ZFC$ that contains $x$ and $A$.

\subsubsection{Completeness}  
\label{section random complete}

$\bar{\RR}_\alpha$ denotes the forcing that consists of all Borel codes for subsets of $2^\alpha$ ordered by $\leq$. 
The quotient of $\bar{\RR}_\alpha$ by $=_\mu$ with the operations induced by $\vee$, $\wedge$ and $-$ is a Boolean algebra.

A 
forcing is called \emph{complete} if every subset has a supremum. 
We will show that $\bar{\RR}_\alpha$ is complete. 
To this end, we associate to every $A\in \bar{\RR}_\alpha$ its \emph{footprint} 
$\foot_A=\langle \foot_{A,t} \mid t\in 2^{(\alpha)}\rangle$, where 
$$\foot_{A,t}=\frac{\mu(A\cap N_t)}{\mu(N_t)}$$ 
denotes the relative measure. 
Let $\foot_A \leq \foot_B$ if $\foot_{A,t} \leq \foot_{B,t}$ for all $t\in 2^{(\alpha)}$. 
Note that $A\leq B$ if and only if $\foot_A \leq \foot_B$. 
If $A\leq B$, then clearly $\foot_{A,t}\leq \foot_{B,t}$ for all $t\in 2^{(\alpha)}$. 
If conversely $A\not\leq B$, then for any $\epsilon>0$ there exists some $t\in 2^{(\alpha)}$ with $\foot_{A,t}>1-\epsilon$ and $\foot_{B,t}<\epsilon$ by Lebesgue's density theorem. 
Thus $\foot_{A,t} > \foot_{B,t}$. 


We next define density points of subsets of $2^\alpha$. 
For countable ordinals, this definition agrees with the usual one. 
The quantifiers in the next definition range over $[\alpha]^{<\omega}$. 

\begin{definition} 
\label{def density} 
Suppose that $\vec{\foot}=\langle \foot_s \mid s\in 2^{(\alpha)} \rangle$ is a sequence in $\RR$ and 
 $x\in 2^\alpha$. 
\begin{enumerate-(1)} 
\item 
\label{def density 1} 
For any $\epsilon>0$, $x$ is called an \emph{$\epsilon$-density point} of $\foot$ if $$\exists s\ \forall t\supseteq s\ 
\foot_t>1-\epsilon.$$ 

\item 
\label{def density 2} 
$x$ is called a \emph{density point} of $\foot$ if it is an $\epsilon$-density point of $\foot$ for all $\epsilon\in \QQ^+$. 

\end{enumerate-(1)} 
The $\alpha$-Borel code for the set of density points of $\foot$ induced by \ref{def density 1} and \ref{def density 2} is denoted $D(\foot)$. 
\end{definition} 

Any $\alpha$-Borel code can be reduced to a Borel code as follows. 

\begin{definition} 
The \emph{reduct} $\red(A)$ of an $\alpha$-Borel code $A$ is the following Borel code, defined by induction on the rank: 
\begin{enumerate-(1)} 
\item 
If $A$ is a code of rank $0$, then $\red(A)=A$. 

\item 
If $A_0$ is the canonical code for the complement of $A_1$, then $\red(A_0)$ is the canonical code for the complement of $\red(A_1)$. 

\item 
If $A$ is the canonical code for the union of $\vec{A}=\langle A_i \mid i<\alpha\rangle$, 
work in $L[A]=L[\vec{A}]$. 
Let $I$ be the $L[A]$-least countable subset of $\alpha$ such that for all $j<\alpha$: 
$$\mu \bigl(\bigcup_{i\in I} \red (A_i) \bigr)=\mu\bigl(\red (A_j) \cup \bigcup_{i\in I} \red (A_i) \bigr).$$ 
$\red(A)$ is the canonical code for $\bigcup_{i\in I} \red (A_i)$. 

\end{enumerate-(1)} 
\end{definition} 

Note that $\red(A)\in L[A]$. 
By induction on the rank, we have $\red(A)=_\mu A$ in any outer model where $\alpha$ is countable, for any $\alpha$-Borel set $A$. 
This is used in the next construction. 
Suppose that $X$ is a subset of $\bar{\RR}_\alpha$. 
To construct a least upper bound, we first form the least upper bound of the footprints: 
let $\foot_{X,t}=\sup_{A\in X} \foot_{A,t}$ for each $t\in 2^{(\alpha)}$ and 
$$\foot_X=\langle \foot_{X,t} \mid t\in 2^{(\alpha)}\rangle.$$ 

\begin{lemma}\ 
\label{random supremum} 
\begin{enumerate-(1)} 
\item 
\label{random supremum 1}
In any outer model $W$ of $V$ such that 
$\alpha$ is countable in $W$, 
$D(\mathrm{f}_X)$ 
is a least upper bound for $X$. 

\item 
\label{random supremum 2}
$\bar{\RR}_\alpha$ is complete. 
More precisely, for any subset $X$ of $\bar{\RR}_\alpha$ the reduct of $D(\foot_X)$ is a least upper bound for $X$. 

\end{enumerate-(1)} 
\end{lemma}  
\begin{proof} 
\ref{random supremum 1}: 
We work in $W$. 
Since $\alpha$ is countable, $D(\foot_A)$ is Borel. 
To see that it is an upper bound, suppose that $A\in X$. 
We have $D(\foot_A)=_\mu A$ by Lebesgue's density theorem. 
Since $\foot_A \leq \foot_X$, $D(\foot_A)\leq D(\foot_X)$ as required. 
To see that $D(\foot_X)$ is least, suppose $B\in \bar{\RR}_\alpha$ is an upper bound of $X$. 
Again $D(\foot_B)=_\mu B$ by Lebesgue's density theorem. 
We have $\foot_A \leq \foot_B$ for all $A\in X$, since $A\leq B$. 
Thus $\foot_X\leq \foot_B$ and therefore, $D(\foot_X)\leq D(\foot_B)$ as required. 

\ref{random supremum 2}: 
Let $A$ denote the reduct of $D(\foot_X)$. 
In any outer model $W$ where $\alpha$ is countable, $A=_\mu D(\foot_X)$. 
Since $D(\foot_X)$ is a least upper bound for $X$ in $W$ by \ref{random supremum 1}, $A$ is a least upper bound for $X$ in $V$. 
\end{proof}

\begin{remark}\ 
\begin{enumerate-(1)} 
\item 
If $X$ is closed under finite unions, then the footprint of the least upper bound of $X$ is precisely $\foot_X$. 
This is clear if $X$ is countable. 
In general, this can be seen by passing to a generic extension where $X$ is countable. 

\item 
One can avoid the use of outer models in the proof of Lemma \ref{random supremum} 
by a direct calculation that the reduct of $D(\foot_X)$ is a least upper bound for $X$. 
If $A$ and $B$ are codes for unions $\bigcup_{i<\alpha} A_i$ and $\bigcup_{i<\alpha} B_i$ with $\red(A_i)\leq \red(B_i)$ for all $i<\alpha$, one shows $\red(\bigcup_{i<\alpha}A_i)\leq \red(\bigcup_{i<\alpha}B_i)$ and a similar statement for intersections.  

\end{enumerate-(1)} 
\end{remark} 

We now pass to a subforcing $\RR_\alpha$ of $\bar{\RR}_\alpha$ whose definition is absolute to inner models. 
This is not the case for $\bar{\RR}_\alpha$. 
This absoluteness is used to show that $\RR_\alpha$ is $\omega^{[+]}$-c.c. 

\begin{definition} 
The \emph{random algebra} $\RR_\alpha$ on $\alpha$ generators is the subforcing of $\bar{\RR}_\alpha$ that consists of those $A$ such that $\supp(A)$ is countable in $L[A]$. 
\end{definition} 


We also write $\RR$ for \emph{random forcing} $\RR_\omega$. 
The next lemma shows that $\RR_\alpha$ and $\bar{\RR}_\alpha$ are essentially the same forcings, since $\RR_\alpha$ meets every equivalence class in $\bar{\RR}_\alpha$ with respect to $=_\mu$. 

\begin{lemma} 
\label{representatives for random algebras} 
There is an $\OD$ function $F\colon \bar{\RR}_\alpha\rightarrow \RR_\alpha$ that picks a representative in each equivalence class in $\bar{\RR}_\alpha$. 
\end{lemma} 
\begin{proof} 
Let $A \in \bar{\RR}_{\alpha}$ and $[A]$ denote the equivalence class of $A$. 
Let $G\colon \bar{\RR}_\alpha \rightarrow \bar{\RR}_\alpha$ denote the function that sends $A$ to the reduct of $D(\foot_{[A]})$. 
There is a formula that defines $G(A)$ in $L[A]$ from $A$ and $\alpha$ by the definition of $G$. 
$G(A)$ is equivalent to $A$, 
since $G(A)$ is a least upper bound for $[A]$ by Lemma \ref{random supremum} \ref{random supremum 2}. 
Let $A_i:=G^i(A)$ for each $i\in\omega$. 
We have $\omega_1^{L[A_i]}\geq\omega_1^{L[A_{i+1}]}$, since $A_{i+1}\in L[A_i]$. 
Let $n\in\omega$ be least with $\omega_1^{L[A_n]}=\omega_1^{L[A_{n+1}]}$. 
By the definition of $G$, the support of $A_{n+1}=G(A_n)$ is countable in $L[A_n]$ and thus in $L[A_{n+1}]$. 
Hence $A_{n+1}\in \RR_\alpha$ as required. 
\end{proof}



%
%

\begin{theorem}\ 
\label{random complete} 
\begin{enumerate-(1)} 
\item 
\label{random complete 1} 
$\RR_\alpha$ is complete.

\item 
\label{random complete 2} 
$\RR_\alpha$ is locally complete. 

\item 
\label{random complete 3} 
$\RR_\alpha$ is locally $\omega^{[+]}$-c.c. 

\item 
\label{random complete 4} 
$\RR_\alpha$ is uniformly narrow. 

\end{enumerate-(1)} 
\end{theorem} 
\begin{proof} 
\ref{random complete 1}: 
By Lemmas \ref{random supremum} and \ref{representatives for random algebras}. 

\ref{random complete 2}: 
The condition in Lemma \ref{lc equiv} \ref{lc equiv 2} holds by completeness of $\RR_\alpha$ and Lemma \ref{representatives for random algebras}. 

\ref{random complete 3}: 
$\RR_\alpha$ is provably c.c.c. in $\ZFC$. 
Since the definition of $\RR_\alpha$ is absolute to inner models, it is $\omega^{[+]}$-c.c. 

\ref{random complete 4}: 
By Lemma \ref{cc lc imply narrow}. 
\end{proof}

\subsubsection{Borel codes versus sets} 
\label{section - Borel codes versus Borel sets} 

We call a subset of $2^\omega$ \Bore\ if is contained in the $\sigma$-algebra generated by basic open sets. In models of $\DC$, the usual definition of random forcing via \Bore\ sets is isomorphic to ours, since every \Bore\ set has a Borel code. 
However, in the model in \cite[Theorem 10.6]{MR0396271} and in Gitik's model from \cite[Theorem I]{MR576462}, every set of reals is \Bore. 
Thus there exist \Bore\ sets without Borel codes. 
We now show that it suffices for this that $\omega_1$ is singular. 

\begin{remark} 
If $\omega_1$ is singular, then there exists a \Bore\ set without a Borel code. 
Towards a contradiction, suppose that 
every \Bore\ set has a Borel code. 
Fix a cofinal sequence $\vec{\alpha}=\langle \alpha_n\mid n\in\omega\rangle$ in $\omega_1$. 
It can be shown by induction that for all countable $\alpha$, the set $\WO_\alpha$ of codes for $\alpha$ is \Bore. 
Since 
$$ B =\{\langle 0\rangle^n{}^\smallfrown \langle1\rangle^\smallfrown x\mid n\in\omega,\  x\in \WO_{\alpha_n}\}$$ 
is \Bore, it has a Borel code $A$ by assumption. 
One can then construct a sequence $\vec{A}=\langle A_n \mid n\in\omega \rangle$ of Borel codes for the sets $\WO_{\alpha_n}$ from $A$. 

There exists a function that sends each Borel code $A'$ for a Borel set $B'$ to a subtree $T$ of $2^{<\omega}\times \omega^{<\omega}$ with $p[T]=B'$. 
For instance, take the tree that searches for an assignment of {\tt true} and {\tt false} to each location in the Borel code. 
We may assume $T$ is pruned by successively removing nodes without successors. 
The leftmost branch $(x,y)$ of $T$ yields an element $x$ of $B'$. 

By applying this to $\vec{A}$, we obtain a sequence $\vec{x}=\langle x_n \mid n\in\omega\rangle$ with $x_n \in \WO_{\alpha_n}$ for all $n\in\omega$. 
But this would provide a surjection from $\omega$ to $\omega_1$. 
\end{remark} 

A similar argument shows that the existence of Borel codes for all \Bore\ sets is equivalent to $\AC_\omega$ for the set of those \Bore\ sets with a Borel code.

\subsection{Closed forcings} 
\label{section closed forcings} 

A forcing $\PP$ is called \emph{${<}\kappa$-closed} if every decreasing sequence $\vec{p}=\langle p_i \mid i<\alpha\rangle$ in $\PP$ with $\alpha<\kappa$ has a lower bound in $\PP$. 
In this section, we analyse some ${<}\kappa$-closed forcings and the influence of fragments of the axiom of choice on their properties. 
As an application, we will see that the forcing $\CC_{\omega_1} = \Col(\omega_1 ,2 )$ 
that adds a Cohen subset of $\omega_1$ 
collapses $\omega_1$ if $\DC(2^\omega)$ fails, and therefore virtually all bounded support iterations of uncountable length collapse $\omega_1$.


\subsubsection{Dependent choices} 

In this section, we call a class $R$ a \emph{relation} on a class $A$ if $R$ is a subclass of $A^{{<}\Ord}\times A$. 
An \emph{$\alpha$-chain} in $R$ is a sequence $\vec{x}=\langle x_i \mid i<\alpha\rangle$ with $(\vec{x}{\upharpoonright}i,x_i)\in R$ for all $i<\alpha$, 
and $R$ is called \emph{${<}\gamma$-extendible} if any $\alpha$-chain in $R$ for any $\alpha<\gamma$ has a proper end extension. 

\begin{definition} 
\label{definition DC} 
Suppose that $A$ is a class and $\gamma\in\Ord$. 
\begin{enumerate-(1)} 
\item 
$\DC_\gamma(A)$ denotes the statement: 
\begin{center} 
\begin{quote} 
Any ${<}\gamma$-extendible binary relation $R$ on $A$ contains a $\gamma$-chain. 
\end{quote} 
\end{center} 

\item 
For $\delta\leq\Ord$, $\DC_{{<}\delta}(A)$ denotes the conjunction of $\DC_\alpha(A)$ for all $\alpha<\delta$. 
$\DC_{{\leq}\gamma}(A)$ is defined similarly. 

\item 
$\DC(A)$ denotes $\DC_\omega(A)$. 

\item 
$\DC_\gamma$ denotes $\DC_\gamma(V)$. 
$\DC_{{<}\gamma}$, $\DC_{{\leq}\gamma}$ and $\DC$ are defined similarly. 

\end{enumerate-(1)} 
\end{definition} 

The next lemma collects some useful facts about variants of $\DC$. 
In particular, the last claim shows that all axioms are first-order expressible. 

\begin{lemma} \ 
\label{equiv DC} 
Suppose $A$ and $B$ are classes and $\gamma\in\Ord$. 

\begin{enumerate-(1)}  
\item 
\label{equiv DC basic} 

\begin{enumerate-(i)} 
\item 
\label{equiv DC basic 1} 
$\DC_{\gamma}(A)$ $\Rightarrow$ $\DC_{{\leq}\gamma}(A)$. 

\item 
\label{equiv DC basic 2} 
If $B\leq_s A$, then $\DC_\gamma(A)$ $\Rightarrow$ $\DC_\gamma(B)$. 

\item 
\label{equiv DC basic 3} 
If $A^\gamma\leq_s A$, then 
$\DC_\gamma(A)$ $\Rightarrow$ $\DC_{{<}\gamma^+}(A)$. 

\end{enumerate-(i)}

\item 
\label{equiv DC 1} 
$\DC_{{<}\Ord}$ is equivalent to the 
axiom of choice. 

\item 
\label{equiv DC 2} 
$\DC_\gamma(A)$ follows from $\DC_\gamma(x)$ for all sets $x\subseteq A$. 

\end{enumerate-(1)} 
\end{lemma} 
\begin{proof} 
\ref{equiv DC basic}\ref{equiv DC basic 1}: 
Let $\alpha\leq\gamma$ be least such that $\DC_{\alpha}(A)$ fails for a ${<}\alpha$-extendible relation $R$ on $A$. 
Note that $\alpha$ is a limit ordinal. 
Since there are no $\alpha$-sequences in $R$, the relation is ${<}\gamma$-extendible. 
Applying $\DC_{\gamma}(A)$ yields a contradiction. 

\ref{equiv DC basic 2}: 
Apply $\DC_\gamma(A)$ to $F^{-1}(R)$, where $R$ is a relation on $B$ and $F\colon A\rightarrow B$ is surjective. 

\ref{equiv DC basic 3}: 
For any $\alpha<\gamma^+$ and any ${<}\alpha$-extendible relation $R$ on $A$, let $T$ denote the tree of $\alpha$-sequences on $A$ in $R$. 
$T$ is closed at all levels ${<}\alpha$ in the sense that any sequence $\vec{t}=\langle t_i \mid i<j\rangle$ for $j<\alpha$ and $t_i \in \Lev_i(T)$ has an upper bound in $T$, where $\Lev_i(T)$ denotes the $i$th level of $T$ for $i\in\Ord$. 
Take any cofinal $\cof(\alpha)$-sequence of levels in $T$. 
We can translate the restriction of the tree to these levels to a relation on $A^{<\alpha}$. 
Note that $A^{<\alpha}\leq_s A^\gamma$. 
We thus obtain an $\alpha$-sequence in $T$ from $\DC_\gamma(A^{<\alpha})$ using \ref{equiv DC basic 1} and \ref{equiv DC basic 2}. 
Hence there is an $\alpha$-sequence in $R$. 

\ref{equiv DC 1}: 
See \cite[Theorem 8.1]{MR0396271}. 
 
\ref{equiv DC 2}: 
Suppose $\gamma$ is least such that $\DC_\gamma(A)$ fails for some ${<}\gamma$-extendible relation $R$ on $A$. 
We can replace $R$ by a tree $T$ that is closed at all levels ${<}\gamma$ as in \ref{equiv DC basic}\ref{equiv DC basic 3}. 
We construct $\vec{\alpha}=\langle \alpha_j \mid j<\gamma \rangle$ by induction 
letting $\alpha_j$ be least such that $\Lev_j(T)\cap V_{\alpha_i}$ extends all branches in $T_{<j}:=\bigcup_{i<j}\Lev_i(T)\cap V_{\alpha_j}$ using $\DC_j(A\cap V_{\sup_{i<j}\alpha_i})$. 
$\DC_{\gamma}(A\cap V_{\sup_{i<\gamma}\alpha_i})$ yields an element of $\Lev_\gamma(T)$ and thus a $\gamma$-sequence in $R$. 
\end{proof}

\subsubsection{Cohen subsets and collapses} 
\label{section Cohen and collapses} 

We study the forcing $\CC_\kappa= \Col(\kappa,2)$ for adding a Cohen subset to $\kappa$. 
This is the special case of the standard collapse forcing $\Col(\kappa,\lambda)$ for $\lambda=2$. 
Since $\Col(\kappa,\lambda)$ is not ${<}\kappa$-closed if $\kappa$ is singular, we introduce the following variant. 

\begin{definition} \ 
\begin{enumerate-(1)} 
\item 
$\Col(\kappa,\lambda):=\{p\colon \alpha\rightarrow \lambda\mid  \alpha<\kappa\}$. 
\item 
$\Col_*(\kappa,\lambda):=\{(f,g) \mid f\in \Col(\kappa,\lambda),\ g\colon \dom(f)\rightarrow |\dom(f)| \text{ is bijective}\}$. 
\end{enumerate-(1)} 
$\Col(\kappa,\lambda)$ is ordered by reverse inclusion, while 
$\Col_*(\kappa,\lambda)$ is ordered by reverse inclusion in the first coordinate.\footnote{The second coordinate is irrelevant for the order.} 
\end{definition} 

Any $\Col(\kappa,\lambda)$-generic filter over $V$ induces a $\Col_*(\kappa,\lambda)$-generic filter and conversely, with identical generic extensions. 
However, $\Col_*(\kappa,\lambda)$ is ${<}\kappa$-closed for any successor cardinal $\kappa=\nu^+$. 
To see this, suppose $\vec{p}=\langle (f_\alpha,g_\alpha) \mid \alpha<\lambda\rangle$ is a decreasing sequence in $\Col_*(\kappa,\lambda)$ for some $\lambda\leq \nu$. 
Let $f=\bigcup_{\alpha<\lambda} f_\alpha$. 
Since $\langle g_\alpha\mid \alpha<\lambda\rangle$ yields a bijection $\dom(f)\rightarrow \mu$ for some $\mu\leq\nu$, $\dom(f)< \kappa$. 
Thus $(f,g)$ is a lower bound for $\vec{p}$. 

Recall that a forcing $\PP$ is called \emph{$\lambda$-distributive} if for any sequence $\vec{U}=\langle U_{\alpha}\mid \alpha<\lambda\rangle\in V$ of dense open subsets, $\bigcap_{\alpha<\lambda}U_{\alpha}$ is dense.\footnote{In models of $\ZFC$, a forcing is $\lambda$-distributive if and only if it does not add new $\lambda$-sequences of element of $V$. 
However, this equivalence can fail in models of $\ZF$ by a result of Karagila and Schilhan \cite{karsch2021}.} 
$\PP$ is called  \emph{${<}\kappa$-distributive} if it is  \emph{$\lambda$-distributive} for all $\lambda<\kappa$. 
We will characterise ${<}\kappa$-distributivity of $\CC_{\kappa}= \Col(\kappa,2)$ for successor cardinals $\kappa$.
We first provide new criteria for $\lambda$-distributivity via properties of the generic filter. 
A characterisation via games in the ground model is known \cite[Section 6]{holschturwel2022}. 
For an infinite cardinal $\lambda$, we say that a filter $G$ on $\PP$ is \emph{$(\PP,\lambda)$-generic} over $V$ if for any sequence  $\vec{U}=\langle U_{\alpha}\mid \alpha<\lambda\rangle\in V$ of dense open sets, $G\cap \bigcap_{\alpha<\lambda}U_{\alpha}\neq \emptyset$.


\begin{lemma} 
\label{char distributive} 
Suppose $\PP$ is a forcing and $\lambda\in \Card$. 
The following conditions are equivalent, 
where only \ref{char distributive 3}$\Rightarrow$\ref{char distributive 2} and therefore also \ref{char distributive 3}$\Rightarrow$\ref{char distributive 1} use the additional assumption that $G$ is wellorderable in $V[G]$. \begin{enumerate-(a)} 
\item 
\label{char distributive 1} 
$\PP$ is $\lambda$-distributive. 

\item 
\label{char distributive 3} 
Any $\PP$-generic filter $G$ over $V$ is $\lambda$-closed in $V[G]$. 

\item 
\label{char distributive 2} 
Any $\PP$-generic filter $G$ over $V$ is $(\PP,\lambda)$-generic over $V$. 

\end{enumerate-(a)} 
\end{lemma} 
\begin{proof} 
\ref{char distributive 1}$\Rightarrow$\ref{char distributive 3}: 
We work in $V[G]$ and show that $G$ is $\lambda$-closed. 
It is easy to see that $\lambda$-distributive forcings do not add new $\lambda$-sequences. 
If $\vec{p}=\langle p_i \mid i< \lambda \rangle$ is any sequence in $G$, then $\vec{p}\in V$. 
Since $D_i :=\{q \in \PP \mid  q \leq p_i \vee q \perp p_i \}$ is dense in $\PP$ for each $i < \lambda$ and the sequence $\langle D_i \mid i < \lambda \rangle$ is in $V$, the set $\bigcap_{i< \lambda} D_i$ is dense in $V$. Since $G$ is $\PP$-generic, 
$G$ contains some $p$ in $\bigcap_{i < \lambda} D_i$ and such a $p$ satisfies that $p\leq p_i$ for all $i<\lambda$, as required. 

\ref{char distributive 3}$\Rightarrow$\ref{char distributive 2}: 
Suppose $\vec{U}=\langle U_i\mid i<\lambda\rangle\in V$ is a sequence of dense open subsets of $\PP$ and $G$ is a $\PP$-generic filter over $V$.  
Since $G$ is wellorderable in $V[G]$, we can construct a decreasing sequence $\langle p_i\mid i<\lambda \rangle$ with $p_i\in G\cap U_i$ in $V[G]$. 
By assumption, there exists some $p\in G$ with $p\leq p_i$ for all $i<\lambda$. 
Then $p\in G\cap \bigcap_{i<\lambda} U_i$ as required. 

\ref{char distributive 2}$\Rightarrow$\ref{char distributive 1}: 
Suppose $\vec{U}=\langle U_\alpha\mid \alpha<\lambda\rangle$ is a sequence of dense open subsets of $\PP$. 
For any $p\in \PP$, let $G$ be a $\PP$-generic filter over $V$ that contains $p$. 
Fix a condition $q\in G\cap \bigcap_{\alpha<\lambda}U_\alpha$ by assumption. 
Since $p\parallel q$, let $r\leq p,q$. 
Since each $U_\alpha$ is open, $r\in \bigcap_{\alpha<\lambda}U_\alpha$. 

Note that \ref{char distributive 1}$\Rightarrow$\ref{char distributive 2} is clear and does not need the extra assumption. 
\end{proof}

For instance, if $T$ is a pruned\footnote{I.e. $T_s=\{ t\in T \mid s\subseteq t \vee t\subseteq s \}$ has height $\lambda$ for all $s\in T$.} $\kappa$-Aronszajn tree and the associated forcing $\PP_T$\footnote{$T$ with its reverse order.} preserves regularity of $\kappa$, then $T$ is ${<}\kappa$-distributive by \ref{char distributive 3}$\Rightarrow$\ref{char distributive 1}. 

%
%

\begin{theorem} 
\label{char collapse} 
Suppose that $A$ is any set with $|A|\geq2$, $\lambda\in \Card$ and 
$\PP=\Col(\lambda^+,A)$. 
The following conditions are equivalent: 
\begin{enumerate-(a)} 
\item 
\label{char collapse 1} 
$\DC_\lambda(A^\lambda)$. 

\item 
\label{char collapse 2} 
$\PP$ is $\lambda$-distributive. 

\item 
\label{char collapse 3} 
$\PP$ does not change $V^\lambda$. 

\item 
\label{char collapse 4} 
$\PP$ preserves size and cofinality of all ordinals $\alpha\leq \lambda^+$. 

\item 
\label{char collapse 6} 
$\PP$ preserves $\lambda^+$ as a cardinal. 

\item 
\label{char collapse 5} 
$\PP$ forces that $\lambda^+$ is regular. 

\end{enumerate-(a)} 
The same equivalences hold for $\Col_*(\lambda^+,A)$. 
In both cases, $A$ may be replaced by any set $B$ with $A\leq_s B \leq_s A^{<\lambda^+}$. 
\end{theorem} 
\begin{proof}
The following arguments prove the equivalence of \ref{char collapse 1}-\ref{char collapse 5} for both $\Col(\lambda^+,A)$ and $\Col_*(\lambda^+,A)$. 

\ref{char collapse 1} $\Rightarrow$ \ref{char collapse 2}: 
Since $A^{<\lambda^+}\leq_s A^\lambda$, we have $\DC_{\leq\lambda}(A^{<\lambda^+})$ by Lemma \ref{equiv DC}\ref{equiv DC basic}\ref{equiv DC basic 2} and \ref{equiv DC basic 1}. 
Suppose $\vec{U}=\langle U_i \mid i<\lambda\rangle$ is a sequence of open dense subsets of $\PP$. 
For any $p\in \PP$, let $T$ denote the tree of decreasing sequences $\vec{p}=\langle p_i \mid i<\alpha \rangle$ in $\PP$ with $\alpha<\lambda$, $p_0\leq p$ and $p_i \in U_i$ for all $i<\alpha$. 
By $\DC_{\leq\lambda}(A^{<\lambda^+})$, $T$ is ${<}\lambda$-extendible and has a branch of length $\lambda$. 
Since $\lambda^+$ is regular by $\DC_{\leq\lambda}(A^{<\lambda^+})$, this branch has a lower bound $q\in \PP$. 
Then $q\in \bigcap_{i<\lambda} U_i$. 

\ref{char collapse 2} $\Rightarrow$ \ref{char collapse 3}: 
This is a standard argument. 

\ref{char collapse 3} $\Rightarrow$ \ref{char collapse 1}: 
It is easy to see that $\PP$ adds a wellorder of $A^{<\lambda^+}$. 
One can thus find the required $\lambda$-chain in the generic extension. 
Since $\PP$ does not change $(A^{<\lambda^+})^\lambda$, it exists in $V$. 

\ref{char collapse 3} $\Rightarrow$ \ref{char collapse 4} $\Rightarrow$ \ref{char collapse 6}: 
These implications are clear. 


\ref{char collapse 6} $\Rightarrow$ \ref{char collapse 5}: 
Let $G$ be $\PP$-generic over $V$ and work in $V[G]$. 
Suppose that 
$\vec{\alpha}=\langle \alpha_i\mid i<\lambda\rangle$ is cofinal in $\lambda^{+V}$ for some $\nu\leq\lambda$. 
Since $\PP$ adds a wellorder of $\pow(\lambda)^V$, 
we obtain a sequence $\langle f_i\mid i<\nu\rangle$ of surjections $f_i\colon \lambda\rightarrow \alpha_i$ for $i<\nu$. 
These can be combined to a surjection $f\colon \lambda\rightarrow\lambda^{+V}$, contradicting the assumption. 

\ref{char collapse 5} $\Rightarrow$ \ref{char collapse 2}: 
Let $G$ be $\PP$-generic over $V$ and work in $V[G]$. 
Since $\PP$ adds a wellorder of $A^{<\lambda^+}$, 
both $\Col(\lambda^+,A)$ and $\Col_*(\lambda^+,A)$ are wellorderable in $V[G]$. 
Since $\lambda^{+V}$ is regular in $V[G]$ by assumption, $G$ is ${<}\lambda^{+V}$-closed. 
Then $\PP$ is $\lambda$-distributive by Lemma \ref{char distributive}. 

For the additional claim, note that $A^{<\lambda^+}\leq_s A^\lambda$. 
Thus  $\DC_\lambda(A^\lambda)$, $\DC_\lambda(B^\lambda)$ and $\DC_\lambda(A^{<\lambda^+})$ are equivalent by Lemma \ref{equiv DC}\ref{equiv DC basic}\ref{equiv DC basic 2}. 
The equivalence of \ref{char collapse 1}-\ref{char collapse 5} holds for $\DC_\lambda(A^{<\lambda^+})$ and $\PP=\Col(\lambda^+,A^{<\lambda^+})$. 
Since $A\leq_s B \leq_s A^{<\lambda^+}$, there exist projections $\Col(\lambda^+,A^{<\lambda^+}) \rightarrow \Col(\lambda^+,B^\lambda) \rightarrow \Col(\lambda^+,A^\lambda)$. 
We thus obtain equivalences of \ref{char collapse 1}-\ref{char collapse 5} for $\DC_\lambda(B^\lambda)$ and $\PP=\Col(\lambda^+,B^\lambda)$ from the previous ones. 
\end{proof} 

For $\lambda=\omega$, regularity of $\omega_1$ is not sufficient to obtain the above conditions. 
For instance, in Cohen's first model $\omega_1$ is regular while $\DC(2^\omega)$ fails. 
By Theorem \ref{char collapse}, virtually any bounded support iteration of length $\omega_1$ collapses $\omega_1$ over models where $\DC(2^\omega)$ fails. 
Note that the result does not have an analogue for singular limit cardinals, 
since $\Col(\aleph_\omega,2)$ forces $\aleph_\omega^V$ to be countable in $\ZFC$. 
We finally add a further characterisation via forcing axioms to Theorem \ref{char collapse}. 


\begin{remark} 
\label{char FA} 
$\DC_\kappa$ can be characterised via the forcing axiom $\FA_\kappa({<}\kappa\text{-closed})$ for any $\kappa\in \Card$. 
This axiom states that for any sequence $\vec{D}=\langle D_i\mid i<\kappa\rangle$ of predense subsets of a ${<}\kappa$-cosed forcing $\PP$, there exists a 
filter $g$ on $\PP$ with $g\cap D_i \neq\emptyset$ for all $i<\kappa$. 
$\FA_\kappa({<}\kappa\text{-closed},A)$ for a set $A$ denotes $\FA_\kappa({<}\kappa\text{-closed})$ restricted to forcings $\PP\subseteq A$. 
One can show the following equivalences for any set $A$ with $A^{<\kappa}\leq_s A$: 
\begin{enumerate-(1)} 
\item 
\label{DC and FA 1} 
$\DC_\kappa(A)$ $\Longleftrightarrow$ 
$\DC_{{<}\kappa}(A)$ $+$ $\FA_\kappa({<}\kappa\text{-closed},A)$ 
\item 
\label{DC and FA 2} 
$\AC$ $\Longleftrightarrow$ $\forall \lambda\in (\SucCard \cup \Reg)\ \FA_\lambda({<}\lambda\text{-closed})$ 
\end{enumerate-(1)} 
For instance, $\DC(2^\omega)$ is equivalent to $\FA_\omega(\sigma\text{-closed},2^\omega)$ and thus to the remaining conditions in Theorem \ref{char collapse} for $A=2$ and $\lambda=\omega$. 
Viale proved a related result in \cite[Theorem 1.8]{viale2016useful}. 

\ref{DC and FA 1}: 
Suppose that $\DC_\kappa(A)$ holds. 
To show $\FA_\kappa({<}\kappa\text{-closed},A)$, suppose that $\vec{D}=\langle D_i\mid i<\kappa\rangle$ is a sequence of predense subsets of a ${<}\kappa$-closed forcing $\PP\subseteq A$. 
Let $T\subseteq A^{<\kappa}$ be the tree of sequences $\langle p_i \mid i<\alpha\rangle$ in $\PP$ with $\alpha<\kappa$, $p_i\in D_i$ and $p_j\leq p_i$ for all $i<j<\alpha$. 
An application of $\DC_\kappa(T)$ yields a sequence of length $\kappa$ and thus a filter on $\PP$ as required. 

Conversely, suppose that $\DC_{{<}\kappa}A)$ and $\FA_\kappa({<}\kappa\text{-closed},A)$ hold. 
Suppose that $R$ is a ${<}\kappa$-extendible relation on $A$. 
Let $T\subseteq A^{<\kappa}$ denote the tree of ${<}\kappa$-chains in $R$ ordered by end extension. 
$T$ is ${<}\kappa$-closed by $\DC_{{<}\kappa}(T)$. 
For any $\alpha<\kappa$, let $D_\alpha$ denote the set of chains in $R$ of length at least $\alpha$. 
$\FA_\kappa(T)$ yields a 
branch in $T$ that meets each $D_\alpha$, inducing a $\kappa$-chain in $R$.  

\ref{DC and FA 2}: 
Since $\DC_{{<}\kappa}$ implies $\DC_\kappa$ at singular limits, the claim follows from \ref{DC and FA 1} by induction. 
\end{remark}


\section{Absoluteness principles} 


We begin with a definition of the absoluteness principles described in the introduction. 
Let $M\equiv N$ denote that $M$ and $N$ are elementarily equivalent. 

\begin{definition}\label{GA(C)}
The 
\emph{unrestricted absoluteness principle} $\GA_\cC$ for a class $\mathcal{C}$ of forcings 
states that $V\equiv V[G]$ for any generic extension of $V$ by a forcing in $\cC$. 
\end{definition}

More precisely, $\GA_\cC$ is the scheme of all formulas $\forall \PP\in \cC\ (\varphi \Longleftrightarrow\one \Vdash_\PP \varphi)$, where $\varphi$ ranges over all sentences. 
Our main goal is to understand the consequences of unrestricted absoluteness for various classes $\cC$. 
Note that $\GA_{\CC}$ holds in any Cohen extension and 
$\GA_{\RR}$ holds in any random extension of a model of $\ZFC$. 
Recall $\CC^*$, $\RR_*$ and $\HH^{(*)}$ denote the class of finite support products of Cohen forcings, random algebras, and finite support iterations of Hechler forcings, respectively. 

\subsection{Cohen versus random models} 
\label{switches} 
 
We show that for any sufficiently large cardinal $\kappa$, extensions by $\CC^\kappa$ and $\RR_\kappa$ have different theories. 
The argument is due to Woodin. 

\begin{lemma}[Truss \cite{truss1983noncommutativity}] 
\label{random versus Cohen} 
If $x$ is a Cohen real over $L[y]$ where $y$ is a real, then $y$ is not a random real over $L[x]$. 
\end{lemma} 
\begin{proof} 
The proof relies on the argument showing that a Cohen real $x$ adds a Borel code $A$ for a null set containing all ground model reals. 
It suffices to show $A\in L[x]$, since this implies ground model reals are not random over $L[x]$. 
To see this, let $\vec{B}=\langle B^k_m \mid k,m < \omega\rangle$ be a constructible list 
of codes for all basic open sets with measure at most $2^{-k}$ for each $k\in\omega$. 
Suppose $x$ is Cohen generic over $L[y]$. 
Let $z\in \omega^\omega$ list the distances of successive $n\in \omega$ with $x(n)=1$.\footnote{Thus $\sum_{i<k} z(i)$ is the $k$th $n\in\omega$ with $x(n)=1$.}
Let 
$C_n$ be a Borel code for  $\bigcup_{k\ge n} B^k_{z(k)}$ 
for each $n\in\omega$. 
Since $\mu(C_n)\leq 2^{-(n-1)}$, 
$\bigcap_{n\in \omega} C_n$ is a null set. 
Let $A$ be a Borel code for $\bigcap_{n\in \omega} C_n$ in $L[x]$. 
Since $x$ is Cohen generic over $L[y]$, it is easy to show that for every $n\in\omega$, every real $w\in L[y]$ is an element of $C_n$ as required. 
\end{proof}

We will use that a random real over $V$ is also random over any inner model $M$. 
To see this, note that every maximal antichain $B\in M$ of 
$\RR\cap M$ is maximal in $\RR$, since $B$ is maximal if and only if $\mu(\bigcup B)=1$. 
Note that this holds for $\RR_\kappa$ as well. 
For any $\CC^\kappa$-generic filter $G$ over $V$ and subset $S$ of $\kappa$, let $G_S$ denote the set of all $p\in G$ with $\supp(p)\subseteq S$. 
We use the same notation for $\RR_\kappa$. 

\begin{lemma} 
Suppose $\lambda\leq \kappa$ are infinite cardinals. 
\label{random over intermediate} 
\begin{enumerate-(1)} 
\item (Woodin) 
\label{random over intermediate 1} 
If $H$ is $\CC^\kappa$-generic over $V$ then in $V[H]$, for any subset $A$ of $\kappa$ of size ${<}\lambda$, $H_\lambda$ adds a Cohen real over $L[A]$. 
\item 
\label{random over intermediate 2} 
As in \ref{random over intermediate 1}, but for $\RR_\kappa$ and random reals. 

\end{enumerate-(1)} 
\end{lemma} 
\begin{proof} 
We can assume $\lambda\geq\omega_1$. 
We have shown above that $\CC^\kappa$ and $\RR_\kappa$ preserve cardinals. 

\ref{random over intermediate 1}: 
Since $\CC^\kappa$ is wellorderable, there is a transitive class model $U\subseteq V$ of $\ZFC$ 
such that $H$ is $\CC^\kappa$-generic over $U$ and 
$A\in U[G]$. 
By the $\omega_1$-c.c. of $\CC^\kappa$ in $U$, $A\in U[H_S]$ for some subset $S\in U$ of $\kappa$ with $|S|^U<\lambda$. 
Any coordinate in $\lambda\setminus S$ induces a Cohen real $x$ over $U[H_S]$. 
Since $A\in U[H_S]$, $x$ is a Cohen real over $L[A]$. 

\ref{random over intermediate 2}: 
By Lemmas \ref{lc capturing} and \ref{random complete}, there is a transitive class model $U\subseteq V$ of $\ZFC$ 
such that $G:=H\cap U$ is $(\RR_\kappa\cap U)$-generic over $U$ and 
$A\in U[G]$. 
By the $\omega_1$-c.c. of $\RR_\kappa\cap U$ in $U$, $A\in U[G_S]$ for some subset $S\in U$ of $\kappa$ with $|S|^U<\lambda$. 
Any countably infinite set of coordinates in $\lambda\setminus S$ induces a random real $x$ over $U[G_S]$ by the proofs of \cite[Lemmas 3.1.5, 2.1.6 \& 3.2.8]{Bartoszynski-Judah}.\footnote{These results are formulated for $2^\omega$ but work as well for $2^\kappa$.} 
Since $A\in U[G_S]$, $x$ is a random real over $L[A]$. 
\end{proof} 



\begin{lemma} 
\label{no random over Cohen} 
Suppose $\kappa$ is an uncountable cardinal. 
\begin{enumerate-(1)} 
\item (Woodin) 
\label{no random over Cohen 1} 
If $H$ is $\CC^\kappa$-generic over $V$ then in $V[H]$, there exists an subset $A$ of $\omega_1$ such that there exists no random real over $L[A]$. 

\item 
\label{no random over Cohen 2} 
As in \ref{no random over Cohen 1}, but for $\RR_\kappa$ and Cohen reals. 

\end{enumerate-(1)} 
\end{lemma} 
\begin{proof} 
\ref{no random over Cohen 1}: 
Suppose that $y\in V[H]$ is a random real over $L[H_{\omega_1}]$. 
Since $y$ is a countable subset of $\kappa$ in $V[H]$, by Lemma~\ref{random over intermediate}~\ref{random over intermediate 1}, $H_{\omega_1}$ adds a Cohen real $x$ over $L[y]$. 
By Lemma \ref{random versus Cohen}, $y$ is not random over $L[x]$. 
Therefore $y$ cannot be a random real over $L[H_{\omega_1}]$. 

\ref{no random over Cohen 2}: 
Suppose that $y\in V[H]$ is a Cohen real over $L[H_{\omega_1}]$. 
Since $y$ is a countable subset of $\kappa$ in $V[H]$, by Lemma~\ref{random over intermediate}~\ref{random over intermediate 2}, $H_{\omega_1}$ adds a random real $x$ over $L[y]$.
By Lemma \ref{random versus Cohen}, $y$ is not a Cohen real over $L[x]$. 
Therefore $y$ cannot be a Cohen real over $L[H_{\omega_1}]$, 
\end{proof} 

In fact, the proof of \ref{no random over Cohen 1} shows that in $V[H]$, for any cardinal $\lambda$ with $\omega_1\leq\lambda\leq\kappa$, there exists an subset $A$ of $\lambda$ of size $\lambda$ such that there exists no random real over $L[A]$. 
A similar claim holds for \ref{no random over Cohen 2}. 


In the next theorem, let $\cC$ denote the class of all forcings of the form $\CC^\kappa$ or $\RR_\kappa$ for any $\kappa\in \Card$. 

\begin{theorem}[Woodin] 
\label{paper failure absoluteness} 
$\A_{\cC}$ fails. 
In fact, there is a single switch 
that works for all models of $\ZF$. 
\end{theorem} 
\begin{proof} 
After forcing with $\RR_\kappa$ for any $\kappa\geq\omega_2$, for any subset $A$ of $\omega_1$ there exists a random real over $L[A]$ by Lemma \ref{random over intermediate} \ref{random over intermediate 2}. 
However, this statement is false after forcing with $\CC^\lambda$ for any $\lambda\geq\omega_1$ by Lemma \ref{no random over Cohen} \ref{no random over Cohen 1}. 
An alternative proof works for $\CC^\kappa$ and $\RR_\lambda$ using Lemmas \ref{random over intermediate} \ref{random over intermediate 1} and \ref{no random over Cohen} \ref{no random over Cohen 2}. 
\end{proof} 

%



\subsection{Hartogs numbers} 
\label{generic absoluteness}

Recall that $\CC^*$ denotes the class of finite support products of Cohen forcings and $\RR_*$ denotes the class of all random algebras.
The previous section suggests to study the classes $\CC^*$ and $\RR_*$ separately. 
We will see that each of $\GA_{\CC^*}$ and $\GA_{\RR_*}$ implies that all limit ordinals have countable cofinality. 
To this end, we analyse the Hartogs number $\aleph = \aleph (2^{\omega})$ in generic extensions. 
Recall that the \emph{Hartogs number} $\aleph(x)$ of a set $x$ is defined as the least ordinal $\alpha$ such that $\alpha \not\leq_i x$. We write $\aleph$ for the Hartogs number $\aleph (2^{\omega})$. So $\aleph = \sup \{ \alpha \in \Ord \mid \alpha \le_i 2^{\omega} \}$. 
For a cardinal $\kappa$, let $\kappa^-$ be the cardinal $\kappa$ if $\kappa$ is a limit cardinal and otherwise the cardinal predecessor of $\kappa$. Then $\aleph^- = \sup \{ \lambda \in \Card \mid \lambda \le_i 2^{\omega} \}$.  
Woodin proved Theorem \ref{A implies singular} \ref{A implies singular 2} for the class $\CC^*$ in response to the authors' Remark \ref{original remark}. 
The proofs in this section are extensions of Woodin's argument. 

\begin{remark} 
\label{original remark} 
$\GA_{\CC^*}$ implies that there cannot exist two distinct uncountable regular cardinals. 
To see this, suppose that $\kappa<\lambda$ are the first two uncountable regular cardinals. 
We claim that $\CC^\nu$ forces $\aleph^-=\nu$ for any $\omega$-strong limit cardinal $\nu$ of uncountable cofinality. 
Then we would have $\cof(\aleph^-)=\kappa$, the first uncountable regular cardinal, in generic extensions by $\CC^\nu$ when $\nu$ is an $\omega$-strong limit cardinal of cofinality $\kappa$ while we would have $\cof(\aleph^-)=\lambda$, the second uncountable regular cardinal, in generic estensions by $\CC^\nu$ when $\nu$ is an $\omega$-strong limit cardinal of cofinality $\lambda$, contradicting $\GA_{\CC^*}$. 
If the claim fails, then one can obtain $\nu^+ \leq_i \nu^{\omega}$ by taking a name $\dot{f}$ for an injective function from $\nu^+$ to $2^\omega$ and picking a sequence of nice names for reals in $\HOD_{\dot{f}}$.
Since $\cof(\nu)\geq\omega_1$, 
we have $\nu^{\omega} = \bigcup_{\mu < \nu} \mu^{\omega}$, and then $\nu \leq_i \mu^{\omega}$ for some $\mu<\nu$. 
But this contradicts the fact that $\nu$ is an $\omega$-strong limit. 

Moreover, note that at least two uncountable regular cardinals exist if there exists at least one and $\GA_\cC$ holds for the class $\cC$ of all forcings of the form $\Col(\omega,\kappa)$ or $\Col(\omega,{<}\kappa)$, where $\kappa\in\Card$. 
If $\kappa$ is an uncountable regular cardinal, then $\Col(\omega,{<}\kappa)$ forces that $\omega_1$ is regular and thus $\omega_1$ is regular in $V$ by $\GA_\cC$. 
Then any infinite successor cardinal $\lambda^+$ is regular by $\GA_\cC$, since otherwise $\Col(\omega,\lambda)$ would force that $\omega_1$ is singular. 
\end{remark}

We now proceed with a more general argument 
that works for instance for the classes $\CC^*$ and $\RR_*$. 
This is based on a proof of Woodin for $\CC^*$. 



\begin{definition} 
\label{def nice} 
$\PP$ is called \emph{nice} 
if for all ordinals $\nu$, if $p\Vdash_\PP \nu\leq_i 2^\omega$ for some $p\in \PP$ then $\nu \leq_i \PP^\omega$. 
\end{definition} 

The idea for this definition is that the required function $\nu \rightarrow \PP^\omega$ sends each $\alpha<\nu$ to a nice name for the $\alpha$th real as in the proof of the next lemma. 


\begin{lemma} 
\label{narrow small} 
Every locally complete locally $\omega^{[+]}$-c.c. forcing $\PP$ 
is nice.
\end{lemma} 
\begin{proof} 
Suppose $p\Vdash_\PP \lambda \leq_i 2^\omega$. 
Then there exists some $q\leq p$ and a sequence $\vec{\sigma}=\langle \sigma_\alpha \mid \alpha<\lambda\rangle$ of $\PP$-names for reals with $q \Vdash \sigma_\alpha \neq \sigma_\beta$ for all $\alpha<\beta<\lambda$. 
Suppose $x$ witnesses that $\PP$ is locally complete. 
We can assume $x$ also witnesses that $\PP$ is locally $\omega^{[+]}$-c.c. and contains $\PP$, $q$ and $\vec{\sigma}$. 
For each $\alpha<\lambda$ and $n<\omega$, let $A_{\alpha,n}$ denote the set of all conditions forcing $n\in \sigma_\alpha$. 
Since $\sup(A_{\alpha,n}\cap \HOD_x)=\sup(A_{\alpha,n})$ by \ref{lc equiv 1} in Lemma \ref{lc equiv}, there exists an antichain $A'_{\alpha,n}$ in $\HOD_x$ with supremum $\sup(A_{\alpha,n})$. 
Let $A'_{\alpha,n}$ be least in $\HOD_x$. 
It is countable in $\HOD_x$ by the $\omega^{[+]}$-c.c. 
Let $\vec{p}_{\alpha,n}=\langle p_{\alpha,n,i} \mid i\in\omega  \rangle$ be the least enumeration in $\HOD_x$ of  $A'_{\alpha,n}$ of order type $\omega$. 
Let 
$$ \tau_\alpha = \{ (\check{n}, p_{\alpha,n,i}) \mid n,i <\omega \}. $$ 
Then $ q \Vdash_\PP \sigma_\alpha=\tau_\alpha$. 
Therefore, the map sending $\alpha<\lambda$ to 
$\vec{p}_\alpha:=\langle p_{\alpha,n,i} \mid n,i\in\omega  \rangle \in \PP^{\omega\times\omega}$ is injective as required. 
\end{proof}



\begin{assumption} 
\label{cond seq} 
$\vec{\PP}=\langle \PP_\kappa \mid \kappa\in \Card\rangle$ denotes a sequence of forcings with the properties: 
\begin{enumerate-(a)} 
\item 
\label{cond seq 1} 
$\PP_\kappa \leq_i \kappa^{\omega}$. 

\item 
\label{cond seq 2} 
$\PP_\kappa$ adds a $\kappa$-sequence of distinct reals. 

\item 
\label{cond seq 3} 
$\PP_\kappa$ is nice. 

\item 
\label{cond seq 4} 
$\PP_\kappa$ is uniformly $\omega$-narrow. 

\end{enumerate-(a)} 
\end{assumption} 

Conditions \ref{cond seq 3} and \ref{cond seq 4} hold if $\PP_\kappa$ is locally complete and locally $\omega^{[+]}$-c.c. by Lemmas \ref{cc lc imply narrow} \ref{cc lc imply narrow 2} and \ref{narrow small}. 
For example, all conditions hold for $\CC^\kappa$ and $\RR_\kappa$. 
Notice that $\PP_\kappa$ preserves all cardinals and cofinalities by \ref{cond seq 4} and Lemma~\ref{narrow pres card}.



\begin{definition} Let $\nu \in \Card$ and $A, B$ be subsets of $\nu^{\omega}$. 
\begin{enumerate-(1)} 
\item 
$B$ \emph{covers} $A$ if for each $x\in A$, there exists some $y\in B$ with $\ran(x)\subseteq \ran(y)$. 

\item 
A subset $B$ of $\nu^{\omega}$ of size $\aleph^-$ is called \emph{minimal} if it is not covered by any subset $A$ of $\nu^{\omega}$ of size ${<}\aleph^-$. 
 
\item 
$\jj$ denotes the least cardinal $\nu$ such that there exists a minimal subset of $\nu^{\omega}$, if this exists. 
\end{enumerate-(1)} 
\end{definition} 

Note that if $\aleph^- = \omega$, then $\jj$ does not exist. On the other hand, if $\aleph^- \ge \omega_1$, then $\jj$ exists and $\omega_1 \le \jj \le \aleph^-$. In particular, if $\aleph^- = \omega_1$, then $\jj = \omega_1$. 


We will analyze the circumstances in which $\jj\geq\aleph^-$ holds in $\PP_\kappa$-generic extensions. 
Note that for any uncountable regular cardinal $\lambda$, there exist arbitrarily large $\omega$-inaccessible cardinals of cofinality $\lambda$.\footnote{Lemma \ref{covering} and thus Theorem \ref{A implies singular}\ref{A implies singular 1} use a weaker condition than $\kappa$ being an $\omega$-strong limit, namely for each $\nu<\kappa$ there exists no injection from $\kappa$ into $\nu^{\omega}$.} 
Woodin proved the next lemma for $\CC^\kappa$. 

\begin{lemma}\ 
\label{covering} 
\begin{enumerate-(1)} 
\item 
\label{covering 1} 
$\one_{\PP_\kappa} \Vdash \aleph=\kappa^+$ for any $\omega$-inaccessible cardinal $\kappa$. 
 
\item 
\label{covering 2} 
$\one_{\PP_\kappa} \Vdash (\aleph=\kappa^+ \Rightarrow \jj\geq\aleph^-)$ for any $\omega$-strong limit cardinal $\kappa$. 


\end{enumerate-(1)} 
\end{lemma} 
\begin{proof} 
\ref{covering 1}: 
Otherwise we have $p \Vdash_{\PP_\kappa} \aleph^->\kappa$ for some $p\in \PP_\kappa$ by \ref{cond seq 2}. 
Since $\PP_\kappa$ is nice by \ref{cond seq 3}, $\kappa^+\leq_i \PP_\kappa^\omega$. 
Since $\PP_\kappa^\omega \leq_i (\kappa^{\omega})^{\omega} \leq_i \kappa^{\omega}$ by \ref{cond seq 1}, $\kappa^+\leq_i \kappa^{\omega}$. 
Since $\cof(\kappa)>\omega$ by assumption, we have $\kappa^{\omega} = \bigcup_{\nu < \kappa} \nu^{\omega}$ and hence $\kappa \leq_i \nu^{\omega}$ for some $\nu<\kappa$, contradicting that $\kappa$ is an $\omega$-strong limit. 

\ref{covering 2}:  
Let $V[G]$ be a $\PP_\kappa$-generic extension  of $V$. 
We work in $V[G]$. 
Suppose that $\aleph = \kappa^+$. Then $\aleph^- = \kappa$. 
Note that $\jj$ exists since $\aleph^-\geq\omega_1$. 
Let $\nu<\kappa=\aleph^-$ and $B$ be a subset of $\nu^{\omega}$ of size $\kappa$. 
We claim that $B$ is not minimal. 
It suffices to find a wellorderable subset $A\in V$ of $\nu^{\omega}$ that covers $B$. 
Since $\kappa$ is an $\omega$-strong limit in $V$, $|A|<\kappa$ follows. 
Fix a bijective function $f\colon \kappa\rightarrow B$. 
Let $\dot{g}$ be a $\PP_\kappa$-name for the function $g\colon \kappa \times \omega\rightarrow \nu$ that sends $(\alpha,n)$ to $f(\alpha) (n)$. 
Let $\dot{f}$ be a $\PP$-name for the bijection $f$ and $p\in \PP$ force that $\dot{g}$ satisfies the definition of $g$ with $\dot{f}$ described in the last sentence.
For each $(\alpha,n) \in \kappa \times \omega$, 
let $D_{\alpha,n}$ denote the set of all conditions ${\leq} p$ in $\PP_\kappa$ that decide $\dot{g}(\alpha)(n)$. 
Define $g_{\alpha,n}\colon D_{\alpha,n} \rightarrow \nu$ such that 
$r\Vdash \dot{g}(\alpha)(n) = g_{\alpha,n}(r)$ for each $r \in D_{\alpha , n}$. 
Then $g_{\alpha,n}$ is a $\parallel$-homomorphism. 
Since $\PP_{\kappa}$ is uniformly $\omega$-narrow by \ref{cond seq 4}, there is a function $h\colon \kappa\times \omega \rightarrow \nu^{\omega}$ such that $h(\alpha ,n)$ enumerates $\ran \bigl( g (\alpha ,n)\bigr)$ for each $(\alpha , n) \in \kappa \times \omega$. Let $\bar{h} \colon \kappa \to (\nu^{\omega})^{\omega}$ be such that $\bar{h} (\alpha) (n) = h (\alpha , n)$. Using a bijection between $\omega$ and $\omega \times \omega$, we may assume $\bar{h} \colon \kappa \to \nu^{\omega}$. 
%
Finally, let $A=\ran(\bar{h})$. Then $A$ is wellorderable and $A$ covers $B = \ran(f)$, since $\ran\bigl(f(\alpha)\bigr)\subseteq \bigcup_{n \in \omega} \ran \bigl(h(\alpha,n )\bigr) = \ran \bigl(\bar{h} (\alpha)\bigr)$ for all $\alpha<\kappa$.  
\end{proof} 


The next lemma is a stronger version of a lemma of Woodin for $\CC^\kappa$. 

\begin{lemma} 
\label{covering fails} 
If $\nu\in\Card$, $p\in \PP_\nu$ forces that $\aleph$ is a successor cardinal and $\aleph> (\aleph^V)^+$, then $p \Vdash_{\PP_\nu} \jj\leq \nu$. 
\end{lemma} 
\begin{proof} 
We may assume that $p$ forces that $\lambda = \aleph^-$ for some $\lambda \in \Card$. 
Since $p$ forces that $\aleph$ is a successor cardinal and $\PP$ is nice by \ref{cond seq 3}, $\lambda \leq_i \PP_\nu^\omega$. 
Since $\PP_\nu^\omega \leq_i ( \nu^{\omega})^{\omega} \leq_i \nu^{\omega}$ by \ref{cond seq 1}, 
$\lambda\leq_i \nu^{\omega}$. 
We claim that any subset of $\nu^{\omega}$ of size $\lambda$ in $V$ is minimal in $V[G]$ for any $\PP_{\kappa}$-generic $G$ over $V$ with $p \in G$. 
To see this, fix an injective function $f \colon \lambda \rightarrow \nu^{\omega}$ in $V$. 
If $\ran(f)$ is not minimal, 
then there exists some cardinal $\mu<\lambda$, a $\PP_\nu$-name $\dot{g}$ for a function $\dot{g}\colon \mu \rightarrow \nu^{\omega}$ such that some $q\leq p$ forces that $\ran(\dot{g})$ covers $\ran(f)$. 
The next step is similar to the proof of Lemma \ref{covering} \ref{covering 2}. 
For each $(\alpha,n) \in \mu \times \omega$, 
let $D_{\alpha,n}$ denote the set of all conditions ${\leq} q$ in $\PP_\kappa$ that decide $\dot{g}(\alpha)(n)$. 
Define $g_{\alpha,n}\colon D_{\alpha,n} \rightarrow \nu$ such that 
$r\Vdash \dot{g}(\alpha)(n) = g_{\alpha,n}(r)$ for each $r \in D_{\alpha , n}$. 
Then $g_{\alpha,n}$ is a $\parallel$-homomorphism. 
Since $\PP_{\kappa}$ is uniformly $\omega$-narrow by \ref{cond seq 4}, there is a function $h\colon \kappa\times \omega \rightarrow \nu^{\omega}$ such that $h(\alpha ,n)$ enumerates $\ran \bigl( g (\alpha ,n)\bigr)$ for each $(\alpha , n) \in \kappa \times \omega$. Let $\bar{h} \colon \kappa \to (\nu^{\omega})^{\omega}$ be such that $\bar{h} (\alpha) (n) = h (\alpha , n)$. Using a bijection between $\omega$ and $\omega \times \omega$, we may assume $\bar{h} \colon \kappa \to \nu^{\omega}$. 
Let $\vec{A}=\langle A_{\alpha} \mid \alpha<\mu\rangle$, where 
$$A_{\alpha}=\{ \gamma < \lambda \mid \ran\bigl(f (\gamma) \bigr)\subseteq \ran \bigl(\bar{h}(\alpha)\bigr) \}.$$ 
Since $\ran \bigl(\bar{h}(\alpha)\bigr)$ is countable, $\otp(A_{\alpha})<\aleph^V$ for all $\alpha< \mu$. 
Hence $|\bigcup \vec{A}| \le \max (\aleph^V,\mu)$. 
Since $\bigcup \vec{A}=\lambda$ by the choice of $\dot{g}$, we have 
$\lambda\le  \max (\aleph^V,\mu)$. 
However, by assumption, $p$ forces that $\lambda = \aleph^-$ and $\aleph > (\aleph^V)^+$, so $\lambda > \aleph^V$. 
Since $\mu < \lambda$, we now have $\lambda > \max (\aleph^V,\mu)$, contradicting $\lambda\le  \max (\aleph^V,\mu)$.  
\end{proof}

The next theorem was proved by Woodin for the class $\CC^{*}$. 

\begin{theorem} 
\label{A implies singular} 
Suppose $\GA_{\vec{\PP}}$ holds. 
\begin{enumerate-(1)} 
\item 
\label{A implies singular 1} 
$\one_{\PP_\kappa} \Vdash \aleph >\kappa^+$ for any $\omega$-strong limit cardinal $\kappa$. 

\item 
\label{A implies singular 2} 
All infinite cardinals have countable cofinality. 

\end{enumerate-(1)} 
\end{theorem} 
\begin{proof} 
\ref{A implies singular 1}: 
Towards a contradiction, suppose that there exists an $\omega$-strong limit cardinal $\kappa$ with $p \Vdash_{\PP_\kappa} \aleph =\kappa^+$ for some $p\in \PP_\kappa$.  
By Lemma \ref{covering} \ref{covering 2}, $p \Vdash_{\PP_\kappa} \jj\geq\aleph^-$. 
It suffices that $\jj<\aleph^-$ holds in some $\PP_\lambda$-generic extension for some $\lambda\in \Card$, as this would contradict $\GA_{\vec{\PP}}$. 
To see this, pick any successor cardinal $\lambda\ge \aleph^{+}$. 
Since $\PP_\kappa$ preserves all cardinals by \ref{cond seq 4} and Lemma~\ref{narrow pres card}, $p$ forces that $\aleph$ is a successor cardinal and $\aleph^-$ is a limit cardinal. 
By $\GA_{\vec{\PP}}$, the same holds for $\PP_\lambda$. 
Since $\lambda$ is a successor cardinal and $\PP_\lambda$ preserves all cardinals, $\one_{\PP_\lambda} \Vdash \aleph > \aleph^->\lambda \ge (\aleph^V)^+$. 
Since $\one_{\PP_\lambda}$ forces that $\aleph$ is a successor cardinal, $\one_{\PP_\lambda}$ forces $\jj\leq\lambda<\aleph^-$ by Lemma \ref{covering fails} as required. 

\ref{A implies singular 2}: 
Otherwise there exists an $\omega$-inaccessible cardinal. 
Then \ref{A implies singular 1} and Lemma \ref{covering} \ref{covering 1} provide contradictory conclusions. 
\end{proof} 

\begin{remark} 
Theorem \ref{A implies singular} works if in Assumption \ref{cond seq}, \ref{cond seq 1}  is replaced by 
$\PP_\kappa \leq_i \pow_{\omega^{[+]}}(\kappa)$ and \ref{cond seq 4}  by the statement that $\PP_\kappa$ is $\omega$-narrow. 
This avoids the assumption that $\PP_\kappa$ is uniformly narrow. 
The above proof can be modified to this setting by replacing $\nu^\omega$ by $\pow_{\omega_1}(\nu)$ in the definition of $\jj$ and adapting the proofs of Lemmas \ref{covering} and \ref{covering fails}. 
\end{remark} 


\begin{corollary} 
\label{no embeddings} 
There exists some $\kappa\in\Card$ such that there is no elementary embedding $j\colon V \rightarrow V[G]$ in any outer model of $V[G]$ for any $\PP_\kappa$-generic filter $G$ over $V$. 
\end{corollary} 
\begin{proof} 
It the claim fails, then $\GA_{\vec{\PP}}$ holds in $V$. 
Moreover, $j$ necessarily moves an ordinal by \ref{cond seq 2} if $\kappa>\aleph^-$. 
Thus $\mathrm{crit}(j)$ is regular in $V$, contradicting Theorem \ref{A implies singular} \ref{A implies singular 2}. 
\end{proof}

\subsection{The bounding and dominating numbers} 
\label{section bounding} 

The \emph{bounding number} $\bb$ and \emph{dominating number} $\dd$ are defined as the least cardinal $\kappa$ such that there exists an unbounded, respectively dominating, family in $\omega^\omega$ of size $\kappa$. 
They need not exist in choiceless models. 

If $p\in \PP$ and $\sigma$ is a $\PP$-name, we write $p\vdash \sigma$ if $p$ decides the value of $\sigma$, i.e., $p\Vdash \sigma=\check{x}$ for some $x\in V$. 

\begin{lemma} 
\label{density of conditions deciding support} 
Suppose $\vec{\PP}=\langle \PP_\alpha,\dot{\PP}_\alpha,\PP_\gamma\mid \alpha< \gamma\rangle$ is a finite support iteration and $\vec{f}=\langle \dot{f}_\alpha \mid \alpha<\gamma\rangle$ is a sequence of $\PP_\alpha$-names with $\one_{\PP_\alpha}\Vdash \dot{f}_\alpha\colon \dot{\PP}_\alpha \rightarrow \check{V}$ for all $\alpha<\gamma$. 
Then
$$  \QQ:= \{ p\in \PP_\gamma \mid \forall \alpha \in \supp(p)\  p{\upharpoonright}\alpha \vdash \dot{f}_\alpha(p(\alpha)) \} $$ 
is dense in $\PP_\gamma$. 
\end{lemma} 
\begin{proof} 
Fix a wellorder $\leq^*$ of $[\gamma]^{<\omega}$ and $p_0\in \PP_\gamma$. 
We construct the following for some $k\in\omega$: 
\begin{enumerate-(i)} 
\item 
$\vec{s}=\langle s_n\mid n\leq k\rangle$ with $s_n\in [\gamma]^{<\omega}$. 

\item 
$\vec{P}=\langle P_n\mid n<k \rangle$ with $P_n\subseteq \PP_\gamma^{(p_0)}:=\{p\in \PP_\gamma\mid p\leq p_0\}$. 

\item 
$\vec{\alpha}=\langle \alpha_n \mid n<k\rangle$ strictly decreasing. 
\end{enumerate-(i)} 
For all $n<k$, we will have $\max(s_n)=\alpha_n$ and for all $p\in P_n$: 
 
\begin{enumerate-(a)} 
\item 
$p\leq p_0$ 
\item 
$\supp(p)=s_n\cup \{\alpha_i\mid i<n\}$ 
\item 
$p{\upharpoonright}\alpha_i$ decides $\dot{f}_{\alpha_i}(p(\alpha_i))$ for all $i< n$. 
\end{enumerate-(a)} 
Let $P_0=\{p_0\}$, $s_0=\supp(p_0)$ and $\alpha_0=\max(s_0)$. 
In the successor step, suppose that $s_n$, $P_n$ and $\alpha_n$ have been constructed. 
Let $s_{n+1}$ be the $\leq^*$-least support of a condition $r\leq q{\upharpoonright}\alpha_n$ in $\PP{\upharpoonright}\alpha_n$ deciding $\dot{f}_{\alpha_n}(q(\alpha_n))$ for some $q\in P_n$. 
If $s_{n+1}\neq \emptyset$, let $\alpha_{n+1}=\max(s_{n+1})$ and let $P_{n+1}$ be the set of conditions $q\in \PP_\gamma$ such that $\supp(q)\cap \alpha_n=s_{n+1}$, $q{\upharpoonright}\alpha_n$ decides $\dot{f}_{\alpha_n}(q(\alpha_n))$ and there is some $p\in P_n$ with $q{\upharpoonright}\alpha_n \leq p{\upharpoonright}\alpha_n$, 
$\supp(p)\setminus \alpha_n = \supp(q)\setminus \alpha_n$ and 
$p(\alpha)=q(\alpha)$ for all $\alpha\geq\alpha_n$. 
Since $\vec{\alpha}$ is strictly decreasing, there is some $n\in\omega$ with $s_{n+1}=\emptyset$. 
Let $k=n+1$. 
There is some $q \in P_n$ with support $\{\alpha_0,\dots,\alpha_n\}$ such that $q{\upharpoonright}\alpha_i$ decides $\dot{f}_{\alpha_i}(q(\alpha_i))$ for all $i\leq n$. 
Thus $q\in \QQ$ and $q\leq p_0$. 
\end{proof} 

Suppose that $\PP$ and $\SSS$ are forcings. 
Recall that a $\leq$-homomorphism $g\colon \PP\rightarrow \SSS$ is called a \emph{projection} if $\ran(g)$ is dense in $\SSS$ and for all $p\in \PP$ and all $t\leq g(p)$, there exists some $p'\leq p$ with $g(p')\leq t$. 
We call a function $g\colon \PP\rightarrow \SSS$ a \emph{$\perp$-projection} if $g$ is simultaneously a $\perp$-homomorphism and projection. 
For the next lemma, suppose $\vec{\PP}=\langle \PP_\alpha,\dot{\PP}_\alpha, \dot{f}_\alpha, \PP_\delta  \mid \alpha< \delta\rangle$ is a sequence such that $\vec{\PP}=\langle \PP_\alpha,\dot{\PP}_\alpha, \PP_\delta\mid \alpha< \delta\rangle$ is a finite support iteration and $\one_{\PP_\alpha}$ forces ``$\dot{f}_\alpha\colon \dot{\PP}_\alpha\rightarrow \check{\SSS}$ is a $\perp$-projection'' for all $\alpha<\delta$. 

\begin{lemma} 
\label{perp projection} 
For each $\gamma\leq\delta$, there exists a dense subset $\QQ_\gamma$ of $\PP_\gamma$ and a $\perp$-projection $g_\gamma\colon \QQ_\gamma \rightarrow \SSS^\gamma$, where $\SSS^{\gamma}$ is the finite support product of $\SSS$ of length $\gamma$. 
\end{lemma} 
\begin{proof} 
For $\gamma\leq\delta$, let $\QQ_\gamma$ denote the set of $p\in \PP_\gamma$ such that $p{\upharpoonright}\alpha$ decides $\dot{f}_{\alpha}(p(\alpha))$ for all $\alpha\in \supp(p)$. 
$\QQ_\gamma$ is dense in $\PP_\gamma$ by Lemma \ref{density of conditions deciding support}. 
Let 
$g\colon \QQ_\gamma\rightarrow \SSS^\gamma$ with $g(p)=t$, where $\dom(t)=\supp(p)$ and 
$\forall \alpha\in \supp(p)\ \ 
p{\upharpoonright}\alpha \Vdash_{\PP_\alpha} \dot{f}_\alpha(p(\alpha))=\check{t}(\alpha)$. 
Clearly, $g$ is a $\leq$-homomorphism. 

\begin{claim*} 
$g$ is a $\perp$-homomorphism. 
\end{claim*} 
\begin{proof} 
Suppose that $p,q\in \QQ_\gamma$ with $g(p)\parallel g(q)$. 
Let $t=g(p)$ and $u=g(q)$. 
Let $\alpha_0,\dots,\alpha_{n-1}$ enumerate $\supp(p)\cup\supp(q)$ in increasing order and let $\alpha_n:=\gamma$. 
It suffices to show $p \parallel q$. 
To this end, we define a sequence $\langle s_i\mid i\leq n\rangle$ with 
$\dom(s_i)=\alpha_i$, $s_i\leq p{\upharpoonright}\alpha_i, q{\upharpoonright}\alpha_i$ and $s_j{\upharpoonright}\alpha_i \leq s_i$ for $i\leq j\leq n$. 
Then $s_n\leq p,q$ witnesses $p\parallel q$. 
Let $s_0=\mathbf{1}{\upharpoonright}\alpha_0$. 
Suppose $s_i$ is defined, where $i<n$. 
If $\alpha_i\in \supp(p)\setminus \supp(q)$, let $s_{i+1}=s_i^\smallfrown \langle p(\alpha_i)\rangle^\smallfrown \mathbf{1}^{(\alpha_i,\alpha_{i+1})}$. 
The case $\alpha_i\in \supp(q)\setminus \supp(p)$ is similar. 
Now suppose $\alpha_i\in \supp(p)\cap \supp(q)$. 
By the inductive hypothesis, 
$s_i \Vdash_{\PP_{\alpha_i}} \dot{f}_{\alpha_i}(p(\alpha_i))=\check{t}(\alpha_i)=\check{u}(\alpha_i)=\dot{f}_{\alpha_i}(q(\alpha_i))$. 
Hence $s_i\Vdash_{\PP_{\alpha_i}}p(\alpha_i) \parallel q(\alpha_i)$. 
Pick a $\PP_{\alpha_i}$-name $\sigma$ and some $s\leq s_i$ with $s\Vdash_{\PP_{\alpha_i}} \sigma\leq p(\alpha_i), q(\alpha_i)$. 
Then $s_{i+1}=s^\smallfrown \langle\sigma\rangle^\smallfrown \mathbf{1}^{(\alpha_i,\alpha_{i+1})}$ is as required. 
\end{proof} 

\begin{claim*}
$g$ is a projection. 
\end{claim*} 
\begin{proof} 
To see that $g$ is a projection, let $g(p)=t$ and $s\leq t$. 
Let $p_0:=p$ and $\alpha:=\min(\dom(s))$. 
Since $\one_{\PP_\alpha}$ forces that $\dot{f}_\alpha$ is a projection, there exists $q_\alpha\leq p_0{\upharpoonright}\alpha$ in $\QQ_\alpha$ and a $\PP_\alpha$-name $\sigma$ with $q_\alpha\Vdash_{\PP_\alpha} ``\sigma \leq p(\alpha)$ and $\dot{f}_\alpha(\sigma) \le s(\alpha)$'' and $q_{\alpha} \vdash \dot{f}_{\alpha} (\sigma)$. 
Let $p_\alpha=q_\alpha^\smallfrown\langle\sigma\rangle^\smallfrown \langle p_0 (\beta) \mid \alpha<\beta<\gamma \rangle \in \QQ_\gamma$. 
Repeating this process for all other $\alpha\in \dom(s)$ up to $\beta:=\max(\dom(s))$ yields some $p_\beta\leq p$ with $g(p_\beta) \leq s$. 

It remains to show that $\ran(g)$ is dense in $\SSS^\gamma$. 
To see this, take any $t\in \SSS^\gamma$ and let $\alpha:=\min(\dom(t))$. 
Pick some $q_\alpha\in \QQ_\alpha$ and a $\PP_\alpha$-name $\sigma_\alpha$ with $q_\alpha \Vdash \dot{f}_\alpha(\sigma_\alpha)\leq t(\alpha)$ and $q_{\alpha} \vdash \dot{f}_{\alpha} (\sigma_{\alpha})$. 
Let $p_\alpha:=q_\alpha^\smallfrown\langle\sigma_{\alpha} \rangle^\smallfrown \one^{(\alpha+1,\gamma)}\in\QQ_\gamma $. 
Repeating this process for all other $\alpha\in \dom(t)$ up to $\beta:=\max(\dom(t))$ yields some $p_\beta\in\QQ_\gamma$ with $g(p_\beta) \leq t$. 
\end{proof}

Thus $g\colon \PP_\gamma\rightarrow \SSS^\gamma$ is a $\perp$-projection as required. 
\end{proof} 

If $\kappa>\omega$ is regular, then one can force $\bb=\dd=\kappa$: 

\begin{theorem} 
\label{bounding regular} 
Suppose that $\kappa$ is a cardinal of uncountable cofinality. 
Then $\HH^{(\kappa)}$ forces $\bb=\dd=\cof(\kappa)$. 
\end{theorem} 
\begin{proof} 
Suppose that $G$ is $\HH^{(\kappa)}$-generic over $V$. 
$\HH^{(\kappa)}$ does not change the value of $\cof(\kappa)$ 
by Corollary \ref{iteration linked}. 
Since the iteration adds a dominating real in each step, it suffices to show that every real $V[G]$ is an element of $V[G{\upharpoonright}\gamma]$ for some $\gamma<\kappa$. 

Work in $V$ and note that the usual linking function for Hechler forcing sending a tree to its stem is a $\perp$-projection to the forcing $\Func_{<\omega}(\omega,\omega)$ . 
Let $\dot{f}_\alpha$ denote the canonical $\HH^{(\alpha)}$-name for this linking function. 
Let $\QQ_\kappa$ denote the dense subset of $\HH^{(\kappa)}$ and $g\colon \QQ_\kappa\rightarrow \Func_{<\omega}( \kappa \times \omega,\omega)$ the $\perp$-projection given by Lemma \ref{perp projection}. 

Let $\sigma$ be a $\QQ_\kappa$-name for a real. 
We will find some $\gamma<\kappa$ with $\sigma^G\in V[G{\upharpoonright}\gamma]$. 
For each $n<\omega$, let $D_n$ denote the dense set of all $p\in \QQ_\kappa$ that decide whether $n\in\sigma$. 
Since $g$ is a projection, $g[D_n]$ is a dense subset of $\CC^\kappa$. 
Let $A_n$ be the least maximal antichain in $g[D_n]$ in $\HOD_{\{\sigma\}}$ for each $n<\omega$. 
Since $g$ is a $\perp$-homomorphism, $\bar{A}_n:=D_n\cap g^{-1}[A_n]$ is predense in $\QQ_\kappa$. 
By the c.c.c. of $\Func_{<\omega}( \kappa \times \omega,\omega)$ in $\HOD_{\{\sigma\}}$, there exists some $\gamma<\kappa$ with 
$$\bigcup_{n\in\omega,\ p\in \bar{A}_n} \supp(p)=\bigcup_{n\in\omega,\ t\in A_n} \dom(t)\subseteq \gamma.$$ 


For any $p\in \HH^{(\kappa)}$, 
write $p|\gamma:= (p{\upharpoonright}\gamma)^\smallfrown 1^{(\gamma,\kappa)}$. 
Fix $n\in\omega$ and let $\theta_n$ denote any of the formulas $n\in \sigma$ and $n\notin \sigma$.

\begin{claim*} 
If $p\in \QQ_\kappa$ and $p\Vdash_{\HH^{(\kappa)}} \theta_n$, then $p|\gamma\Vdash_{\HH^{(\kappa)}} \theta_n$. 
\end{claim*} 
\begin{proof} 
Otherwise some $q\leq p|\gamma$ in $\QQ_\kappa$ forces $\neg\theta_n$. 
Since $\bar{A}_n$ is predense in $\QQ_\kappa$, there exists some $r\in \bar{A}_n$ with $g(q)\parallel g(r)$. 
Since $g$ is a $\perp$-homomorphism, $q \parallel r$. 
Putting this together, we have $r\Vdash \neg\theta_n$ since $q\Vdash \neg\theta_n$, $r$ decides $\theta_n$ and $q \parallel r$. 
Moreover $p\parallel r$, since $\supp(r)\subseteq\gamma$, $q\leq p|\gamma$ and $q\parallel r$. 
We now obtain $r\Vdash \theta_n$, since $p\Vdash \theta_n$, $r$ decides $\theta_n$ and $p\parallel r$. 
This is a plain contradiction. 
\end{proof} 
The previous claim yields $\sigma^G\in V[G{\upharpoonright}\gamma]$, as desired. 
%
%
%
\end{proof}



The next result shows that if all uncountable cardinals are singular,  
then any iteration of Hechler forcing of uncountable cardinal length 
forces $\bb=\omega_1$. 
By a \emph{uniform iteration of nontrivial forcings}, we mean a sequence 
$\vec{\PP}=\langle \PP_\alpha,\dot{\PP}_\alpha, \dot{p}_{\alpha,i}, \PP_\gamma  \mid \alpha< \gamma,\ i<\omega\rangle$ such that $\vec{\PP}=\langle \PP_\alpha,\dot{\PP}_\alpha, \PP_\gamma\mid \alpha< \gamma\rangle$ is an iteration and $\one_{\PP_\alpha}\Vdash \dot{p}_{\alpha,i} \perp_{\dot{\PP}_\alpha} \dot{p}_{\alpha,i}$ for all $\alpha<\gamma$ and $i<j<\omega$. 

\begin{theorem} 
\label{bounding} 
Suppose $\nu\geq\omega_1$ is multiplicatively closed and has countable cofinality. 
Any uniform iteration $\PP_\nu$ of nontrivial forcings with finite support of length $\nu$ forces: 
\begin{enumerate-(1)} 
\item 
\label{bounding 1} 
$\bb=\omega_1$ if $\PP_\nu$ preserves $\omega_1$. 
\item 
\label{bounding 2} 
$\dd\geq |\nu|$ if $\PP_\nu$ preserves $|\nu|$ and $\dd$ exists in the extension.  
\end{enumerate-(1)} 
\end{theorem} 
\begin{proof} 
\ref{bounding 1}: 
We will construct an unbounded family of size $\omega_1$. 
Let $\vec{\PP}=\langle \PP_\alpha,\dot{\PP}_\alpha, \dot{p}_{\alpha,i}, \PP_{\nu} \mid \alpha< \nu,\ i<\omega\rangle$ denote the iteration. 
We can assume that for each $\alpha<\nu$, $\one_{\PP_\alpha}$ forces that $\langle \dot{p}_{\alpha,i} \mid i<\omega \rangle$ is not maximal by omitting a condition. 
Fix a cofinal strictly increasing sequence $\vec{\alpha}=\langle \alpha_n\mid n\in\omega\rangle$ in $\nu$. 
Fix an injective function $f\colon \nu\times \omega \rightarrow \nu$ such that $f(\alpha,n)\geq \alpha_n$ for all $\alpha<\nu$ and $n<\omega$. 
Such a function can be obtained by thinning out 
$\p{\upharpoonright} (\nu\times\nu) \colon \nu\times \nu\rightarrow\nu$, using that $\nu$ is multiplicatively closed. 
One can easily write down a sequence $\langle \dot{x}_\alpha\mid \alpha<\nu\rangle$ of $\PP_\nu$-names such that 
$\one_{\PP_\alpha}\Vdash \dot{x}_\alpha(n)=i+1$ if $i<\omega$ is unique with $\dot{p}_{f(\alpha,n),i}\in \dot{G}_{f(\alpha,n)}$ 
and 
$\one_{\PP_\nu} \Vdash \dot{x}_\alpha(n)=0$ if no such $i<\omega$ exists. 
Suppose that $G$ is $\PP_\nu$-generic over $V$ and work in $V[G]$. 
Let $x_\alpha=\dot{x}_\alpha^G$ for all $\alpha<\nu$. 
Write $y\leq_* z$ if $\exists m\ \forall n\geq m\ y(n)\leq z(n)$ and define the \emph{trace} of $x\in \omega^\omega$ as
$$\tr(x):=\{\alpha<\nu\mid x_\alpha\leq_* x\}.$$ 
If $\tr(x)$ is countable, then $x$ does not bound $\langle x_\alpha \mid \alpha<\nu \rangle$. 
Since $\PP_\nu$ preserves $\omega_1$, the next claim shows that $\langle x_\alpha \mid \alpha<\omega_1 \rangle$ is unbounded. 

\begin{claim*} 
$\tr(x)$ is countable for all $x\in \omega^\omega$. 
\end{claim*} 
\begin{proof} 
We will partition $\tr(x)$. 
Let $\dot{x}$ be a $\PP_\nu$-name with $\dot{x}^G=x$. 
It is easy to write down a $\PP_\nu$-name $\dot{g}$ such that $\one_{\PP_\nu}$ forces that 
$\dot{g}(\alpha)$ is the least $i<\omega$ with $\forall j\geq i\ \dot{x}_\alpha (j)\leq \dot{x}(j)$. 
For all $n,i<\omega$, let 
$$A_n^{(i)}:= \{ \alpha<\nu \mid \exists p\in G{\upharpoonright} \alpha_n\ \ p^\smallfrown \mathbf{1}^\nu\Vdash_{\PP_\nu} \dot{g}(\alpha)=i \}.$$ 
Let $A_n=\bigcup_{i<\omega} A_n^{(i)}$ and note that $\tr(x)=\bigcup_{n<\omega} A_n$. 

It suffices to show that each $A_n^{(i)}$ is finite. 
Towards a contradiction, suppose that there exist $n,i<\omega$ such that $A_n^{(i)}$ is infinite.  
Work in $V[G{\upharpoonright}\alpha_n]$ and note that $A_n^{(i)}\in V[G{\upharpoonright}\alpha_n]$. 
Fix a strictly increasing sequence $\langle \beta_k\mid k<\omega\rangle$ in $A_n^{(i)}$. 
Work in $V[G]$. 
Since $\beta_k\in A_n^{(i)}$, we have $\dot{g}^G(\beta_k)= i$ and 
hence $x_{\beta_k}(i')\leq x(i')$ for all $k<\omega$ and all $i'\geq i$. 
Let $j:=\max(i,n)$. 
Since $f(\beta_k,j)\geq \alpha_j \geq \alpha_n$ for all $k<\omega$ as $j\geq n$, 
$\langle x_{\beta_k}(j)\mid k<\omega\rangle$ is $\Col(\omega,\omega)$-generic over $V[G{\upharpoonright}\alpha_n]$. 
Hence $x_{\beta_k}(j)> x(j)$ for some $k<\omega$. 
This contradicts the fact that $x_\alpha(j)\leq x(j)$ for all $\alpha\in A^{(i)}_n$ as $j\geq i$. 
\end{proof} 

\ref{bounding 2}: 
Suppose $p\in \PP_\nu$ forces $\langle \dot{y}_\alpha \mid \alpha< \mu \rangle$ 
is a dominating family in $\omega^\omega$ for some cardinal $\mu<|\nu|$. 
Let $G$ be $\PP_\nu$-generic over $V$ with $p\in G$ and work in $V[G]$. 
Define $\langle x_\alpha \mid \alpha<\nu\rangle$ as in the proof of \ref{bounding 1}.  
The  proof of \ref{bounding 1} provides a function that sends each name $\dot{y}$ for an element of $ \omega^\omega$ to an enumeration of $\tr(\dot{y}^G)$ with order type at most $\omega$. 
Hence $|\bigcup_{\alpha<\mu}\tr(\dot{y}_\alpha^G)|\leq \mu<|\nu|$. 
Since $\PP_\nu$ preserves $|\nu|$, not every $x_\alpha$ is dominated by $\langle \dot{y}_\alpha^G \mid \alpha<\mu \rangle$. 
But this contradicts our assumption. 
\end{proof} 

Note that any iteration of Hechler forcing of length $\omega_1$ with finite support preserves $\omega_1$ by Corollary \ref{iteration linked}. 
If $\omega_1$ is singular, then the previous theorem shows that $\HH^{(\omega_1)}$ forces $\bb=\omega_1$. 
This contrasts the fact that the bounding number is regular in $\ZFC$. 
If $\omega_1$ is regular, then $\HH^{(\omega_1)}$ forces $\bb=\omega_1$ by Theorem \ref{bounding regular}. 
Note that the previous theorem for $\kappa\geq\omega_2$ separates the bounding and dominating numbers. 
In this case, $\bb=\omega_1$ and $\dd$ is either at least $\kappa$ or does not exist. 

Recall that $\HH^{(*)}$ denotes the class of finite support iterations of Hechler forcing whose length is any infinite cardinal. 

\begin{corollary} 
\label{abs H singular}
$\GA_{\HH^{(*)}}$ implies that all infinite cardinals have countable cofinality. 
\end{corollary} 
\begin{proof} 
Suppose there exists an uncountable regular cardinal $\kappa$.
Then $\HH^{(\kappa)}$ forces $\bb=\dd$ by Theorem \ref{bounding regular}. 
On the other hand, since $\HH^{(\aleph_\omega)}$ preserves cardinals, $\HH^{(\aleph_\omega)}$ forces $\bb \neq \dd$ by Theorem \ref{bounding}, which contradicts $\GA_{\HH^{(*)}}$.
\end{proof} 


\subsection{Gitik's model} 
\label{section Gitiks model}

In this section, we show that the principles $\GA_{\CC^*}$, $\GA_{\RR_*}$ and $\GA_{\HH^{(*)}}$ fail in Gitik's model from \cite{MR576462} where all infinite cardinals have countable cofinality. 

\subsubsection{Review of Gitik's construction}\label{Gitik}
All results of this subsection are due to Gitik \cite{MR576462}, while some notation is taken from \cite{MR3274975}.\footnote{The version using \cite{MR3274975} yields the same model as Gitik's. 
The two forcings have isomorphic dense subclasses.} 
We assume that in the ground model $V$, $\ZFC$ holds and there is a proper class of strongly compact cardinals. 
We further assume that there is no regular limit of strongly compact cardinals.\footnote{This additional assumption do not increase the consistency strength of the axiom system.}. 
We further assume there is a predicate for a global wellorder on $V$ with order type $\Ord$. 
This can be added by the pretame class forcing $\Add(\Ord,1)$ without adding new sets.\footnote{This produces a model of Bernays-G\"odel class theory $\mathsf{BG}$. 
More formally, the following proofs can be translated to statements in $V$ about $\Add(\Ord,1)$-names forced by $\one_{\Add(\Ord,1)}$.} 
Let $\vec{\kappa}=\langle \kappa_\xi \mid \xi \in \Ord \rangle$ be the increasing enumeration of all strongly compact cardinals. 
In the resulting model, the closure of $\vec{\kappa}$ will equal the class of all uncountable cardinals. 

For each inaccessible cardinal $\alpha$ and each regular cardinal $\kappa<\alpha$, fix a bijection $ \pow_\kappa(\alpha) \rightarrow \alpha$. 
Let 
$\iota_{\kappa,\alpha}$ 
denote the induced bijection between their power sets. 

\begin{notation}
We distinguish the following types of $\alpha \in \text{Reg}$: 
\begin{itemizenew}
\item[(0)] 
If $\alpha \in [\omega, \kappa_0 )$, 
let $\cfdash (\alpha):=\alpha$.
\item[(1)] 
If $\alpha \geq \kappa_0$ and there is a largest $\kappa_\xi \leq \alpha$,\footnote{I.e., $\alpha \in [\kappa_\xi, \kappa_{\xi+1})$}
let $\cfdash (\alpha):=\alpha$.

\begin{enumerate-(a)}
\item
If $\alpha > \kappa_\xi$ is inaccessible, let $\Psi_\alpha$ be any fine ultrafilter on $\pow_{\kappa_{\xi}}(\alpha)$ and define 
$\Phi_\alpha :=\iota_{\kappa_\xi,\alpha}[\Psi_\alpha]$. 

\item 
If $\alpha > \kappa_\xi$ is accessible, let $\Phi_\alpha$ be any $\kappa_\xi$-complete uniform ultrafilter on $\alpha$. 
\end{enumerate-(a)}

\item[(2)] 
If $\alpha\geq\kappa_0$ and there is no largest strongly compact cardinal ${\leq}\alpha$, let $\alpha'$ denote the largest (singular) limit of strongly compacts ${<}\alpha$ and 
$\cfdash (\alpha):= \cf (\alpha')$. 
Fix a strictly increasing sequence 
$\langle \kappa^\alpha_\nu \mid \nu <\cfdash (\alpha) \rangle$ 
of strongly compact cardinals ${>}\cfdash(\alpha)$ with supremum $\alpha'$. 

\begin{enumerate-(a)}
\item 
If $\alpha$ is inaccessible, let $\Psi_{\alpha, \nu}$ be a fine ultrafilter on $\pow_{\kappa^\alpha_\nu}(\alpha)$ for each $\nu < \cfdash (\alpha)$ and 
$\Phi_{\alpha, \nu}:=\iota_{\kappa^\alpha_\nu,\alpha}[\Psi_{\alpha, \nu}]$. 

\item 
If $\alpha$ is accessible, let  $\Phi_{\alpha, \nu}$ be any $\kappa^\alpha_\nu$-complete uniform ultrafilter on $\alpha$ for each $\nu<\cfdash (\alpha) $. 

\end{enumerate-(a)}
\end{itemizenew}
\end{notation}

We will define the forcing $\PP$ successively via the next definitions. 
For $t \subseteq \text{Reg} \times \omega \times \text{Ord}$, write 
$
\dom(t) := \{ \alpha \in \text{Reg} \mid \exists m\ \exists \gamma \, (\alpha ,m , \gamma) \in t \}
$. 
First, let $P_1$ denote the set of all finite subsets $t$ of $\text{Reg} \times \omega \times \text{Ord}$ such that for every $\alpha \in \text{dom}(t)$, $t(\alpha ) := \{ (m,\beta) \mid (\alpha , m , \beta) \in t \}$ is an injective function from a finite subset of $\omega$ to $\alpha$.  

\begin{definition}
\label{def P2} 
Let $P_2$ denote the set of all $t \in P_1$ such that: 
\begin{enumerate-(a)}
\item 
For every $\alpha \in \text{dom}(t)$, $\cfdash \alpha \in \text{dom}(t)$ and $\text{dom} \bigl(t (\cfdash \alpha ) \bigr) \supseteq \text{dom}\bigl(t(\alpha)\bigr)$. 

\item 
\label{def P2 2} 
If $(\alpha_0, \ldots , \alpha_{n-1})$ is the increasing sequence enumeration of $\text{dom}(t) \setminus \kappa_0$, then there exist $m \ge 1$ and  $j \le n-1$ such that: 
\begin{enumerate-(i)}
\item 
For each $k < j$, $\text{dom}\bigl( t (\alpha_k) \bigr) = m+1$. 

\item 
For each $k$ with $j\leq k<n$,  $\dom \bigl( t(\alpha_k)\bigr) = m$.
\end{enumerate-(i)}
\end{enumerate-(a)}
\end{definition}
The values $m$ and $\alpha_j$ in \ref{def P2 2} are unique for $t$ and can thus be denoted $m (t):= m $ and $\alpha (t) := \alpha_j$. 
Note that $\bigl(\alpha (t), m(t)\bigr)$ is the point that needs to be filled in next in order to extend $t$. 
For any $t\in T\subseteq P_2$, let 
$$\mathrm{Suc}_T (t) := \{ \beta \mid t \cup \{ \bigl( \alpha (t), m (t), \beta \bigr)\} \in T\}.$$ 

\begin{definition}
Let $P_3$ denote the set of all pairs $(s, T)$ such that $s\in T \subseteq P_2$ and: 
\begin{enumerate-(a)}
\item 
For every $t \in T$: 
\begin{enumerate-(i)}
\item 
Either $t \supseteq s$ or $t \subseteq s$. 

\item 
$\text{dom}(t) = \text{dom}(s)$. 

\item 
\emph{(tree-like)}
If $t = r \cup \{ \bigl( \alpha (r) , m (r) , \beta) \}$, 
then $r \in T$. 
\end{enumerate-(i)}

\item 
For every $t \in T$ with $s \subseteq t$: 
\begin{enumerate-(i)} 
\item 
If $\alpha (t)$ is of type 1, 
then 
$
\mathrm{Suc}_T (t) \in \Phi_{\alpha (t)}. 
$

\item 
If $\alpha (t)$ is of type 2 
and $m (t) \in \text{dom} \bigl( t(\cfdash (\alpha (t)) ) \bigr)$, then 
$$
\text{Suc}_T (t) \in \Phi_{\alpha (t), t\bigl( \cfdash (\alpha (t)) \bigr) \, (m(t))}.
$$
\end{enumerate-(i)} 

\end{enumerate-(a)}
$s$ is called the \emph{stem} of $(s,T)$. 
\end{definition}

Let $\PP$ denote $P_3$ with the following partial order. 
For $(r, R), (t,T)\in P_3$, let $(r, R) \le (t,T)$ if 
$r {\upharpoonright} \kappa_0 \supseteq t {\upharpoonright} \kappa_0$, 
$R {\upharpoonright} \bigl( \text{dom}(t) \setminus \kappa_0 \bigr) \subseteq T$ 
and $\dom(r) \supseteq \dom(t)$.  
Let $I$ denote the class of finite subsets of $\Reg$ closed under $\cfdash$. 
For any $s\in I$, let 
$$\PP_s := \{ (t, T ) \in P_3 \mid  \text{ dom}(t) \subseteq s\}.$$ 
%
We have $(t, T) {\upharpoonright} \PP_s := \bigl(t {\upharpoonright} s, \{ u {\upharpoonright} s \mid u \in T\}\bigr) \in P_3$ for all $(t,T) \in P_3$ and $s\in I$ \cite[Lemma 2.4]{MR576462}. 
The next lemma shows how $\PP_s$ can be factored. 

\begin{lemma}\cite[Theorem 2.5]{MR3274975}
\label{two step}
For any $s\in I$ and any strongly compact $\kappa_\xi \in s$, $\PP_s$ is forcing equivalent\footnote{I.e., the Boolean completions are isomorphic.} to a forcing of the form $\PP_{s \cap \kappa_{\xi}} \ast \dot{\QQ}$, where $\PP_{s \cap \kappa_{\xi}}$ forces that $\dot{\QQ}$ does not add any bounded subset of $\kappa_\xi$. 
\end{lemma}

Let $D:=\{ (\alpha,n,\beta) \in \text{Reg} \times \omega \times\Ord \mid \beta<\alpha  \}$. 
In the next definition, we work with proper class functions to simplify the notation. 
One can obtain a formally correct definition by restricting each function to its support. 

\begin{definition}\ 
\begin{enumerate-(1)}
\item 
Let $\mathcal{G}$ be the group of permutations 
$g\colon D \rightarrow D$\footnote{We use the domain $D$ instead of $\text{Reg} \times \omega \times\Ord$ to ensure that $g$ is supported on a set.} such that there exists a sequence of permutations $g_{\alpha}$ of $\alpha\in \Reg$ with: 
\begin{enumerate-(a)}
\item 
$g (\alpha , n , \beta ) = \bigl(\alpha , n , g_{\alpha} (\beta) \bigr)$ for all $\alpha\in \Reg$, $\beta < \alpha$ and $n\in\omega$. 

\item 
$\supp(g_\alpha):= \{ \beta <\alpha \mid g_\alpha(\beta)\neq\beta \}$ is finite for all $\alpha\in \Reg$.   

\item 
$\supp(g):=\{\alpha\in \Reg \mid g_\alpha \neq \id\}$ is finite. 
\end{enumerate-(a)}

\item 
For each $g \in \mathcal{G}$, let $\PP^g \subseteq P_3$ be the set of all $(t,T) \in P_3$ such that $\dom(t) \supseteq \supp(g)$ and for all $\alpha \in \dom(t)$: 
\begin{enumerate-(a)}
\item 
$\dom\bigl( t(\alpha) \bigr) = \text{dom} \bigl(t(\cfdash (\alpha)) \bigr)$. 

\item 
If $\alpha \ge \kappa_0$, then 
$
\ran \bigl( t(\alpha) \bigr) \supseteq \{ \beta \in \text{supp}(g_{\alpha})  \mid \exists r \in T \, \bigl(\beta \in \ran ( r(\alpha) ) \bigr) \}.
$
\end{enumerate-(a)}

\item 
For each $g \in \mathcal{G}$ and $(t,T) \in P_3$, let  
$ g'(t, T) := ( g[t], \{ g[t'] \mid t' \in T\})$. 

\end{enumerate-(1)}
\end{definition}

\begin{lemma} \cite[Lemma 3.2]{MR576462} 
For each $g \in \mathcal{G}$, 
$\PP^g$ is a dense subclass of $\PP$ and 
$g' \colon \PP^g \to \PP^g$ is an automorphism of $\PP^g$. 
\end{lemma}


Since the restriction of $\PP$ to any ordinal is a complete subforcing, 
the Boolean algebra $\BB:=\BB(\PP)$ of regular open subsets of $\PP$ is complete.\footnote{A class Boolean algebra is called \emph{complete} if every subset has a supremum.} 
Let $\iota\colon \PP\rightarrow \BB$ denote the canonical $\leq$- and $\perp$-homomorphism. 
It is easy to see that every $g\in \mathcal{G}$ induces an automorphism $\bar{g}$ of $\BB$ defined by $\bar{g}(U)=g'[U\cap \PP^g]_{\mathrm{down}}$, where $U\in \BB$ is a regular open subset of $\PP$ and $U'_{\mathrm{down}}$ denotes the downward closure of any $U'\subseteq \PP$. 
Using density of $\PP^g$, it can be show that one can recover $g$ from $\bar{g}$ and $\overline{g\circ h}= \bar{g} \circ \bar{h}$ for all $g,h \in \mathcal{G}$. 
Thus $\bar{\mathcal{G}}:=\{ \bar{g} \mid g\in \mathcal{G}\}\cong \mathcal{G}$. 


\begin{definition} \ 
\begin{enumerate-(1)}
\item 
Let $\mathrm{fix}(s) := \{ \bar{g}\in \bar{\mathcal{G}} \mid g\in \mathcal{G},\ \forall \alpha \in s \ g_\alpha = \id_\alpha\}$ for $s\in I$. 
\item Let $\bar{\mathcal{F}}$ be the normal filter of subgroups of $\bar{\mathcal{G}}$ generated by $\mathrm{fix}(s)$ for all $s\in I$. 
\end{enumerate-(1)}
\end{definition}

As usual, for any $\BB$-name $\dot{x}$, let 
$$
\mathrm{sym} (\dot{x}) := \{ \bar{g} \in \bar{\mathcal{G}} \mid \bar{g}(\dot{x}) = \dot{x} \}.
$$
A $\BB$-name $\dot{x}$ is called \emph{symmetric} if $\mathrm{sym}(\dot{x})$ is in $\bar{\mathcal{F}}$.
Let $\mathrm{HS}_{\bar{\mathcal{F}}}$ be the class of all hereditarily symmetric $\BB$-names. 
We say that $s\in I$ \emph{supports} a name $\dot{x} \in \mathrm{HS}_{\bar{\mathcal{F}}}$ if $\mathrm{fix}(s) \subseteq \mathrm{sym} (\dot{x})$. 
The next lemma is important in the proof. 

\begin{lemma}\cite[Lemma 3.3]{MR576462}
Suppose that $\varphi$ is a formula with $n$ free variables 
and $\mathrm{fix}(s)$ supports $\dot{x}_0, \ldots , \dot{x}_{n-1} \in \mathrm{HS}_{\bar{\mathcal{F}}}$, where $s\in I$. 
Then for every $(t, T) \in P_3$: 
$$
(t,T)_\iota \Vdash \varphi (\dot{x}_0 , \ldots , \dot{x}_{n-1}) \iff ((t,T) {\upharpoonright} \PP_s)_\iota \Vdash \varphi (\dot{x}_0 , \ldots , \dot{x}_{n-1}). 
$$
\end{lemma}
Fix a $\BB$-generic filter $G$ over $V$.
The symmetric model is defined as $V(G) = \{ \dot{x}^G \mid \dot{x} \in \text{HS}_{\bar{\mathcal{F}}} \}$. 
Note that for any $s\in I$, 
$G {\upharpoonright} \PP_s := \{ (t,T ) {\upharpoonright} \PP_s \mid (t,T)_\iota \in G\}$ is a $\PP_s$-generic filter over $V$. 

\begin{theorem}\ 
\label{Gitik thm} 
\begin{enumerate-(1)} 
\item 
\label{approximation}
\cite[Lemma 2.4]{MR3274975}
For any set of ordinals $X \in V(G)$, there exists some $s\in I$ with $X \in V[G \upharpoonright \PP_s]$.

\item 
\label{Gitik main}
\cite[Corollary 2.10]{MR3274975}
$V(G)$ is a model of $\ZF$ where every infinite cardinal has cofinality $\omega$. 
Moreover, the closure of $\vec{\kappa}$ equals the class of all uncountable cardinals of $V(G)$.  
\end{enumerate-(1)} 
\end{theorem}


\subsubsection{Absoluteness fails over Gitik's model}
\label{cons of abs for ckappa} 

Recall from Section~\ref{generic absoluteness} the definitions of the Hartogs number $\aleph (x)$ for a set $x$ and of $\kappa^-$ for  $\kappa \in \Card$. Also we wrote $\aleph$ for $\aleph (2^{\omega})$. 
The following variants of $\aleph^-$ will play a key role. 
For any cardinal $\kappa$, we write $\cc_{\kappa}$ for $\bigl(\aleph (\kappa^{\omega})\bigr)^-$. Then $\cc_{\omega} = \cc_2 = \aleph^- $ and  
\begin{align*}
\cc_{\kappa} = \sup \{ \lambda\in \Card \mid \lambda \leq_i \kappa^\omega \}. 
\end{align*} 
We will see that each of $\GA_{\CC^*}$ and $\GA_{\RR_*}$ implies that $\cc_{\kappa}>\kappa$ for some infinite cardinal $\kappa$, while in Gitik's model, $\cc_{\kappa}=\kappa$ for all infinite cardinals $\kappa$. 
We write $\cc_{\omega}^{\CC^{\kappa}}$ for the cardinal $\lambda$ with $\one_{\CC^{\kappa}}\Vdash \lambda = \cc_{\omega}$. 

\begin{lemma}
\label{lem:transition}
$\cc_{\kappa} =\cc_{\omega}^{\CC^{\kappa}}$ for all infinite cardinals $\kappa$. 
\end{lemma} 
\begin{proof} 
To show $\cc_\omega^{\CC^{\kappa}} \le \cc_{\kappa}$, 
suppose that $\dot{f}$  is a $\CC^\kappa$-name and $p\in \CC^\kappa$ forces that 
$\dot{f} \colon \gamma \to 2^{\omega}$ is injective. 
Since $\CC^\kappa$ is nice, 
we have $\gamma\leq_i \kappa^\omega$ and thus $\gamma \leq \cc_{\kappa}$, as desired. 
To show $\cc_\omega^{\CC^{\kappa}} \le \cc_{\kappa}$, it suffices to construct an injective function from $\kappa^\omega \cap V$ into the set of Cohen reals over $V$ in a $\CC^\kappa$-generic extension $V[G]$. 
Let $\langle i_\alpha \mid \alpha<\kappa\rangle$ denote the sequence of the first digits of the Cohen reals added by $G$. 
Let $S$ denote a set of size $\kappa^\omega$ of injective functions in $\kappa^\omega$ such that for any $f,g\in S$ with $f\neq g$, there exist infinitely many $n\in\omega$ with $f(n)\neq g(n)$. 
For each $f\in S$, consider the Cohen real $x_f$ over $V$ defined by $x_f(n)=i_{f(n)}$. 
These reals are pairwise distinct by the choice of $S$. 
\end{proof} 


\begin{lemma}
\label{Gitik c}
In Gitik's model, 
$\cc_{\kappa} = \kappa$ for all infinite cardinals $\kappa$. 
\end{lemma}

\begin{proof}
Suppose that $\kappa$ is an infinite cardinal in $V(G)$ 
and $f \colon \gamma \to {}^{\omega}\kappa$ is an injective function in $V(G)$. 
It suffices to show $\gamma < (\kappa^+)^{V(G)}$. 
By Theorem \ref{Gitik thm}~\ref{Gitik main}, $(\kappa^+)^{V(G)}=\kappa_\xi$ for some $\xi$, where $\kappa_\xi$ is a strongly compact cardinal in $V$. 
We will show $\gamma < \kappa_\xi$. 
By Theorem \ref{Gitik thm}~\ref{approximation}, there is some $s\in I$ with $f\in V[G {\upharpoonright} \PP_s]$. We may assume $\kappa_\xi \in s$. 
By Lemma~\ref{two step}, $\PP_s$ is forcing equivalent to a forcing of the form $\PP_{s \cap \kappa_{\xi}} \ast \dot{\QQ}$, where $\PP_{s \cap \kappa_{\xi}}$ forces that $\dot{\QQ}$ does not add new bounded subsets of $\kappa_\xi$. 
Let $\lambda$ be an inaccessible cardinal in $V$ with $\max(s \cap \kappa_{\xi}) < \lambda < \kappa_\xi$. 
Since $|\PP_{s \cap \kappa_{\xi}}|<\lambda$ and $\PP_{s \cap \kappa_{\xi}}$ forces that $\dot{\QQ}$ does not add new bounded subsets of $\kappa_\xi$, $\lambda$ remains inaccessible in $V[G {\upharpoonright} \PP_s]$. 
Since $f\in V[G {\upharpoonright} \PP_s]$, we have $\gamma < \lambda < \kappa_\xi= (\kappa^+)^{V(G)}$, as desired. 
\end{proof}

Note that $\cc_\omega=\omega$ implies that no cardinal characteristics of the reals exist. 
Using the previous lemma, we obtain the failure of the above absoluteness principles in Gitik's model. 

\begin{theorem} 
\label{A fails Gitik} 
$\GA_{\CC^*}$, $\GA_{\RR_*}$ and $\GA_{\HH_*}$ fail in Gitik's model. 
\end{theorem} 
\begin{proof} 
Each of $\GA_{\CC^*}$ and $\GA_{\RR_*}$ implies $\cc_{\kappa}>\kappa$ for all $\omega$-strong limit cardinals $\kappa$ by Theorem \ref{A implies singular} \ref{A implies singular 1} and Lemma \ref{lem:transition}. 
But in Gitik's model, $\cc_{\kappa}=\kappa$ for all such $\kappa$ by Lemma \ref{Gitik c}. 
Moreover, $\GA_{\HH_*}$ implies that the bounding number equals $\omega_1$ by Theorem \ref{bounding}. 
But in Gitik's model, the bounding number does not exist, since $\cc_\omega=\omega$  by Lemma \ref{Gitik c}. 
\end{proof} 

\section{Open problems}\label{section: open problems} 

The absoluteness principle $\GA_{\Col(\omega,*)}$ for the class 
of collapse forcings $\Col(\omega,\kappa)$ for arbitrary cardinals $\kappa$ is consistent even with $\ZFC$ \cite{golshani2016question}.  
Our main open problem is: 

\begin{problem} 
Are the principles $\GA_{\CC^*}$, $\GA_{\RR_*}$ and $\GA_{\HH^{(*)}}$ consistent? 
\end{problem} 

Our results in Sections \ref{generic absoluteness}-\ref{section Gitiks model} indicate some properties that a model of these principles must have. 
In particular, we argued that they fail in Gitik's model \cite{MR576462}. 
One could aim to show that they also fail in all set generic extensions of this model. 
For $\GA_{\CC^*}$ and $\GA_{\RR_*}$, it suffices that set forcing preserves $(\aleph^-)_\lambda=\lambda$ for sufficiently large cardinals $\lambda$, but the argument for Lemma~\ref{lem:transition} does not work for forcings  which are not wellordered. 
Can the construction of Gitik's model \cite{MR576462} or Gitik's alternative construction of a model where all uncountable cardinals are singular from an almost huge cardinal \cite{MR787954} be adapted to obtain models of the above absoluteness principles? 
If these principles are consistent, we would like to understand the structure of their models. 
For instance, does the $\HOD$ of a model of $\GA_{\CC^*}$ contain large cardinals? 

Section \ref{switches} shows that Cohen and random extensions have different theories. 
This suggests to study the same problem for other classical c.c.c. forcings. 
For instance: 

\begin{problem} 
Do extensions by sufficiently large forcings in $\CC^*$ and $\HH^{(*)}$ necessarily have different theories? 
\end{problem} 

The problem of differentiating theories for different forcing extensions is interesting from the viewpoint of the modal logic of forcing \cite{hamkins2008modal}. 
Switches $\varphi_0, \dots, \varphi_n$ are called \emph{independent} if any choice of truth values for these sentences can be realised in the relevant generic extensions. 

\begin{problem} 
Is it provable in $\ZF$ for each natural number $n$ that there exist $n$ independent switches? 
\end{problem} 
 
Regarding preservation of cardinals in Section \ref{section card pres}, 
we ask whether all $(\theta,1)$-narrow forcings already preserve $\theta^+$.\footnote{A recent argument of Karagila, Schilhan and the second-listed author shows that this is the case for $\theta=\omega$. Furthermore, in July 2024, Elliot Glazer gave an argument showing that every $(\theta,1)$-narrow forcing is $\theta$-narrow. Hence by Lemma~\ref{narrow pres card}~\ref{narrow pres card 2}, all $(\theta, 1)$-narrow forcings preserve $\theta^+$.}  
This would improve Lemma \ref{narrow pres card}. 
Is every $(\theta,1)$-narrow forcing $\theta$-narrow?\footnote{Every $(\theta,1)$-narrow forcing is $\theta$-narrow by an argument of Elliot Glazer in 2024.} 
Is every $\theta$-narrow forcing uniformly $\theta$-narrow?\footnote{Elliot Glazer gave an example of an $\omega$-narrow forcing which is not uniformly $\omega$-narrow in a Solovay model.} 
Is every wellordered $\theta^+$-c.c. forcing uniformly $\theta$-narrow?\footnote{The example in the previous footnote witnesses that not every wellordered c.c.c. forcing is uniformly $\omega$-narrow.} 
Cardinal preservation is also a problem for many classical forcings. 
For example, we do not know whether Sacks forcing preserves $\omega_1$ even if one assumes that $\omega_1$ is inaccessible to reals. 
The problem is to show that every new real has a name in $H_{\omega_1}$. 
If this is the case, then one can show that $\omega_1$ is preserved using capturing as in \cite{castiblanco2021preserving,schschsch}. 
Regarding random algebras in Section \ref{section random forcing}, note that we would have a shortcut for some of our results if $\RR_{\kappa}$ is $\CC^{\kappa}$-linked. 
However, we do not know whether even $\ZFC$ provides a negative answer to the next problem. 

\begin{problem} 
Is $\RR_{\omega_1}$ $\sigma$-linked? 
\end{problem} 

A further natural question regarding $\RR_\kappa$ is whether it is provable in $\ZF$ that $\RR_\kappa$ does not have uncountable antichains for any infinite cardinal $\kappa$. 
Regarding Section \ref{section Cohen and collapses}, 
one can ask for a finer understanding of the effect of $\Add(\kappa,1)$. 
For instance, is it consistent that $\Add(\omega_3,1)$ collapses $\omega_1$ and $\omega_3$ while preserving $\omega_2$? 
It would further be interesting to extend the results about Hechler forcing in Section \ref{section bounding} to obtain a precise characterisation of the bounding number in finite support iterations of Hechler forcing or arbitrary ordinal length even in $\ZFC$. 

Gitik's model in Section \ref{section Gitiks model} is a useful test case to study forcing over choiceless models. 
What are the answers to the following questions for Gitik's model: 
Does every nonatomic $\sigma$-closed forcing collapse $\omega_1$? 
Does the dominating number exist in some generic extension? 
Is the bounding number $\omega_1$ in any generic extension where it exists? 
Does Sacks forcing preserve $\omega_1$? 
Can one increase $\HOD$ by forcing?




\bibliographystyle{amsplain}
\bibliography{daisuke}

\end{document}